\documentclass[12pt, reqno]{amsart}
\usepackage{floatrow}
\usepackage{amssymb, amsfonts, amsthm, amsmath, amscd}
\usepackage{graphicx}   
\usepackage[latin2]{inputenc}
\usepackage{t1enc}
\usepackage[mathscr]{eucal}
\usepackage{indentfirst}
\usepackage{pict2e}
\usepackage{epic}
\numberwithin{equation}{section}
\usepackage{stmaryrd}
\usepackage{scalerel,stackengine}
\usepackage[margin=2.5cm]{geometry}
\usepackage{mathtools}
\usepackage[colorlinks = true,
            linkcolor = blue,
            urlcolor  = blue,
            citecolor = blue,
            anchorcolor = blue]{hyperref}
\usepackage[titletoc]{appendix}
\usepackage[T1]{fontenc}
\usepackage[numbers,square]{natbib}
\usepackage{adjustbox}
\usepackage{dsfont}
\usepackage{times}
\usepackage{enumerate}
\usepackage{setspace}
\usepackage{leftidx}
\usepackage{multirow}

\usepackage[most]{tcolorbox}
\usepackage{pagecolor}
\definecolor{bananamania}{rgb}{0.98, 0.91, 0.71}
\usepackage{tikz-cd}
\usepackage{pgf,tikz}
\usetikzlibrary{math}
\usetikzlibrary{arrows,decorations.markings}
\usetikzlibrary{patterns,intersections}
\usetikzlibrary{calc}
\usetikzlibrary{decorations.pathreplacing}

\newenvironment{myproof}[2]{\paragraph{\textit{Proof of {#1} }{#2}.}}{\hfill$\square$}

\theoremstyle{plain}
\newtheorem{theorem}{Theorem}[section]
\newtheorem{lemma}[theorem]{Lemma}
\newtheorem{corollary}[theorem]{Corollary}
\newtheorem{proposition}[theorem]{Proposition}

\newtheorem{innercustomthm}{Theorem}
\newenvironment{customtheorem}[1]
  {\renewcommand\theinnercustomthm{#1}\innercustomthm}
  {\endinnercustomthm}

\theoremstyle{definition}
\newtheorem{definition}[theorem]{Definition}

\newtheorem{remark}[theorem]{Remark}

\makeatletter
\newcommand*\bulletsmall{\mathpalette\bulletsmall@{.5}}
\newcommand*\bulletsmall@[2]{\mathbin{\vcenter{\hbox{\scalebox{#2}{$\m@th#1\bullet$}}}}}
\makeatother

\def\CC{\mathbb{C}}
\def\DDD{\mathbb{D}}

\def\QQ{\mathbb{Q}}
\def\RR{\mathbb{R}}
\def\ZZ{\mathbb{Z}}
\def\PP{\mathbb{P}}

\def\a{\alpha}
\def\b{\beta}

\def\s{\sigma}

\def\G{\Gamma}
\def\w{\omega}
\def\ra{\rightarrow}
\def\tt{\theta}

\def\O{\Omega}
\def\gg{\mathfrak{g}}
\def\bb{\mathfrak{b}}

\def\ol{\overline}
\def\bl{\bullet}
\def\ll{\lambda}
\def\ecc{q^{-1}_1,\ldots,q^{-1}_k}
\DeclareMathOperator{\conv}{Conv}

\DeclareMathOperator{\fof}{Frac}
\DeclareMathOperator{\spec}{Spec}
\DeclareMathOperator{\sym}{Sym}
\DeclareMathOperator{\rep}{Rep}
\DeclareMathOperator{\ec}{Eff}
\DeclareMathOperator{\pt}{pt}

\DeclareMathOperator{\lie}{Lie}

\DeclareMathOperator{\pr}{pr} 
\DeclareMathOperator{\ev}{ev}
\DeclareMathOperator{\id}{id}
\DeclareMathOperator{\im}{Im}
\DeclareMathOperator{\pd}{PD}

\DeclareMathOperator{\jac}{Jac}
\DeclareMathOperator{\home}{Hom}
\DeclareMathOperator{\res}{Res}
\DeclareMathOperator{\ind}{Ind}

\DeclareMathOperator{\cu}{cu}
\def\rhod{\rho^{\vee}}
\def\ad{\alpha^{\vee}}
\def\Q{Q^{\vee}}

\def\od{\omega^{\vee}}

\def\Bd{B^{\vee}}
\def\Gd{G^{\vee}}
\def\Td{T^{\vee}}
\def\Bmd{B_-^{\vee}}
\def\Bed{B_{e^T}^{\vee}}
\newcommand{\dd}[1]{#1^{\vee}}
\def\Ud{U^{\vee}}

\def\Udd{U^{\vee}(w_Pw_0)}
\def\Pd{P^{\vee}}
\def\Pbd{\ol{P}^{\vee}_-}
\def\pbd{\ol{\mathfrak{p}}^{\vee}_-}

\def\OO{\mathcal{O}}
\def\KK{\mathcal{K}}
\def\BB{\mathcal{B}}

\def\PP{\mathbb{P}}
\def\YY{\mathcal{Y}_P}
\def\YYs{\mathcal{Y}_P^*}
\def\YYY{\mathcal{Y}}
\def\FF{\mathcal{F}}
\def\LL{\mathcal{L}}
\def\ag{\mathcal{G}r}
\def\agl{\mathcal{G}r^{
\lambda}}
\def\agll{\mathcal{G}r^{\leqslant 
\lambda}}

\def\agp{\mathcal{G}r_L^{\leqslant 
w_P w_0(\lambda)}}
\def\HT{H^T_{\bl}}

\def\R{H_T^{\bl}(\pt)}
\def\Hag{H^T_{\bl}(\ag)}
\def\Hmag{H^T_{-\bl}(\ag)}
\def\gm{\mathbb{G}_m}
\def\ga{\mathbb{G}_a}

\def\wl{wt_{\lambda}}

\def\wp{w_Pw_0}
\def\wpd{\dot{w}_P\dot{w}_0}

\def\GB{G^{\vee}/B^{\vee}_-}
\def\tGB{\times_{G^{\vee}/B^{\vee}_-}}
\def\per{\mathcal{P}_{G(\mathcal{O})}(\mathcal{G}r)}
\def\co{\mathbb{X}_{\bl}(T)}
\def\cop{\mathbb{X}^+_{\bl}(T)}
\def\oc{\mathbb{X}^{\bl}(T)}
\def\vp{\varphi}
\def\Jl{\mathcal{J}_{\lambda_0}}
\def\JJ{\mathcal{J}}
\def\MV{S^-_{\mu}}

\def\QHP{QH_T^{\bl}(G/P)}
\def\QHPi{QH_T^{\bl}(G/P)[\ecc]}

\def\sm{\setminus}

\def\MM{\overline{\mathcal{M}}}

\def\PH{{}^p\mathcal{H}^0}
\def\kkk{\Bbbk}

\newcommand{\bk}[1]{\left[#1\right]}
\def\al{\ad}

\newcommand{\fib}[1]{\mathcal{E}^G_{#1}}
\newcommand{\fibl}[1]{\mathcal{E}^L_{#1}}
\newcommand{\ffib}[1]{\mathcal{E}^G_{#1}(G/P)}
\newcommand{\ffibl}[1]{\mathcal{E}^L_{#1}(G/P)}

\def\TTT{\mathcal{T}}
\def\eti{\eta_1}
\def\etii{\eta_2}
\def\etiii{\eta_3}
\def\etiv{\eta_4}
\def\etv{\eta_5}
\def\etvi{\eta_6}
\def\etvii{\eta_7}
\def\etviii{\eta_8}
\def\etix{\eta_9}
\def\vpi{\varphi_1}
\def\vpii{\varphi_2}
\def\vpiii{\varphi_3}
\def\vpiv{\varphi_4}

\begin{document}
\title[On D. Peterson's presentation of quantum cohomology of $G/P$]{On D. Peterson's presentation of quantum cohomology of $G/P$}

\author{Chi Hong Chow}

\begin{abstract} 
We prove in full generality that the $T$-equivariant quantum cohomology of any flag variety $G/P$ is isomorphic to the coordinate ring of a stratum of the Peterson scheme associated to the Langlands dual group scheme $G^{\vee}$. This result was discovered by Dale Peterson but remains unpublished. Our isomorphism is constructed using Yun-Zhu's isomorphism and Peterson-Lam-Shimozono's homomorphism.
\end{abstract}


\maketitle
\section{introduction}\label{1}

\subsection{Main results}\label{1a}
Let $G$ be a simple simply-connected complex algebraic group and $P$ a parabolic subgroup of $G$. We prove the following results expressing the $T$-equivariant quantum cohomology $\QHP$ and its localization in terms of the Langlands dual group scheme $\Gd$. Our coefficient ring is $\ZZ[\ell_G^{-1}]$ where $\ell_G$ is equal to 1 for type ADE, 2 for type BCF and 3 for type G.
\begin{customtheorem}{A}\label{1athm} There exists an isomorphism of $\R[q_1,\ldots,q_k]$-algebras
\[ \Phi:\OO(\YY)\xrightarrow{\sim}\QHP\]
which induces an injective homomorphism
\[ \Phi^*:\OO(\YYs)\ra  \QHPi.\]
The latter is also an isomorphism if we set all equivariant parameters to zero.
\end{customtheorem}

The schemes $\YY$ and $\YY^*$ are defined as follows. Define first $e^T:=e+f\in \dd{\bb}\otimes \R$ where 
\[ e:= \sum_{i=1}^r|\al_i|^2 e_{\ad_i}\quad\text{ and }\quad f:\ad_i\mapsto |\ad_i|^2\a_i,\]
and
\[ \YYY:=\left\{(g\Bmd,h)\in\GB\times\spec\R\left|~g^{-1}\cdot e^T(h)\in \dd{\bb}_-\oplus\bigoplus_{i=1}^r\dd{\gg}_{\ad_i} \right.\right\}.\]
Here $e_{\ad_i}$ is a generator of $\dd{\gg}_{\ad_i}$ and $|\ad_i|^2=1$ or $\ell_G$ depending on whether $\ad_i$ is short or long. Then
\[ \YY:=\YYY\tGB\Ud_{\Pd_-}\wpd\Bmd/\Bmd\]
and 
\[ \YYs:=\YY\tGB\Ud\Bmd/\Bmd.\]
The regular functions 
\begin{equation}\nonumber
\ol{q}_i: (u\wpd\Bmd,h)\mapsto  |\al_i|^{-2}\langle e^*_{-w_0\ad_i},(u\wpd)^{-1}\cdot e^T(h)\rangle
\end{equation}
make $\YY$ and $\YYs$ schemes over $\spec\R[q_1,\ldots,q_k]$. The homomorphisms $\Phi$ and $\Phi^*$ stated in Theorem \ref{1athm} are required to send these regular functions to the quantum parameters $q_1,\ldots,q_k$.

In general, $\OO(\YYs)$ is not isomorphic to $\QHPi$. A counterexample arises already when $G=\mathbf{SL}_{2}$. See Section \ref{1ab}. To obtain a presentation of $\QHPi$, we consider another scheme. Define $\Bed$ to be the centralizer group scheme of $e^T$ in $\Bd\times\spec\R$. This group scheme was introduced by Yun and Zhu \cite{YZ} who proved that $\OO(\Bed)$ is isomorphic to the $T$-equivariant homology $\Hmag$ of the affine Grassmannian $\ag$ of $G$. 
\begin{customtheorem}{B}\label{1athmb} There exists an isomorphism of $\R[q_1,\ldots, q_k]$-algebras
\[ \Phi_{loc}:\OO(\Bed\times_{\Gd} \Ud_-(\wpd)^{-1}\Bmd)\xrightarrow{\sim}\QHPi\]
under which the quantum parameters $q_i$ correspond to the regular functions
\[ \ol{q}_{loc,i}: (u_1(\wpd)^{-1}tu_2,h)\mapsto \ad_i(t).\]
\end{customtheorem}


Theorem \ref{1athmb} has an immediate application to mirror symmetry. Let us take the coefficient ring to be $\CC$. Motivated by earlier work \cite{BCFKS, Givental, JK, Peter}, Rietsch \cite{Rietsch} constructed a Landau-Ginzburg model $(X_P^{\vee},W,\pi,p)$ and conjectured that it is the $T$-equivariant mirror of $G/P$. She gave evidence by proving the following mirror symmetry prediction.
\begin{customtheorem}{C}\label{1athmc} (\cite[Corollary 4.2]{Rietsch}) There exists an isomorphism of $\R[q_1,\ldots,q_k]$-algebras
\[ \jac(X_P^{\vee},W,\pi,p)\simeq \QHPi. \]
Here, $\jac(X_P^{\vee},W,\pi,p)$ is the Jacobi ring which, in the non-equivariant case, is the coordinate ring of the fiberwise critical locus of the superpotential $W$. \hfill $\square$
\end{customtheorem}

However, her proof relies essentially on a version of the Peterson variety presentation which is equivalent to our Theorem \ref{1athmb}, and which, to the best of our knowledge, has remained conjectural since Peterson's announcement (see Section \ref{1h}). Therefore, our work completes the last step of the proof of Theorem \ref{1athmc}.

We will not review the Rietsch mirror in this paper. For more discussion, see our sequel paper \cite{me3} where we enhance the isomorphism from Theorem \ref{1athmc} to an isomorphism of $D$-modules.


\subsection{History}\label{1h} 
In 1997, Peterson \cite{Peter} gave a lecture course at MIT where he announced the following four results about the quantum cohomology of flag varieties:
\begin{enumerate}
\item \textbf{Quantum Chevalley formula}, an explicit formula expressing the quantum product of a Schubert class and a divisor class in terms of other Schubert classes.

\item \textbf{Comparison formula}, an explicit formula expressing the Schubert structure constants of $QH^{\bl}(G/P)$ in terms of those of $QH^{\bl}(G/B)$.

\item \textbf{``Quantum equals affine'' and its generalization to the parabolic case}, which states that the affine Schubert calculus determines the quantum Schubert calculus for $G/P$ completely, and vice versa for $P=B$, through an explicit ring homomorphism from the Pontryagin homology $H_{-\bl}(\ag)$ of the affine Grassmannian $\ag$ of $G$ to $QH^{\bl}(G/P)[\ecc]$.

\item \textbf{Peterson variety presentation}, which is the non-equivariant version of our Theorem \ref{1athm}.
\end{enumerate}

However, he did not publish any of these results. Some years later, complete proofs of the first three results were given by other people: (1) was proved by Fulton and Woodward \cite{FW} and generalized later by Mihalcea \cite{Mihalcea} to the equivariant settings; (2) was proved by Woodward \cite{Woodward}; and (3) was proved by Lam and Shimozono \cite{LS}. The proof in \cite{LS} is combinatorial. Recently, the author of the present paper \cite{me} has given a new proof of (3) which is geometric. Both proofs work for the equivariant settings.

It is (4) which has not been proved in full generality so our work provides the final piece of the puzzle. The case $P=B$ is complete. The quantum cohomology ring $QH^{\bl}(G/B)$ has another presentation formulated in terms of the Toda lattice. It was established by Ciocan-Fontanine \cite{CF}, Givental and Kim \cite{GK} for type A, and Kim \cite{Kim} for the general type. Although it is different from Peterson's presentation, they are equivalent after a change of variables due to Kostant \cite{Kostant}.

Beyond the Borel case, again using known presentations of quantum cohomology \cite{add1, add2, add3, KH1, KH2, ST, Witten}, Rietsch \cite{Pa} proved (4) for Grassmannians and later for all type A flag varieties \cite{Pab}; and Cheong \cite{Pc} for Lagrangian and orthogonal Grassmannians. For minuscule flag varieties, Lam and Templier \cite{LT} proved the $D$-module mirror conjecture and used it to establish the equivariant version of (4) by taking the semi-classical limit and applying a result of Rietsch \cite{Rietsch}. 



\subsection{Example: the rank-one case}\label{1ab} 
As a warm up, let us prove by hand Theorem \ref{1athm} and Theorem \ref{1athmb} for the case $G=\mathbf{SL}_{2}$ and $P=B$. We have $G/P\simeq\PP^1_{\CC}$ and $\R\simeq \ZZ[h]$. The Langlands dual group scheme $\Gd$ is isomorphic to $\mathbf{PGL}_{2}$. Take $\Td\subset\Gd$ to be the diagonal torus and $\Bd$ (resp. $\Bmd$) to be the subgroups of upper (resp. lower) triangular matrices. Then $\GB\simeq\PP^1_{\ZZ}$. Denote by $\Ud$ (resp. $\Ud_-$) the unipotent radical of $\Bd$ (resp. $\Bmd$). Let $\a$ (resp. $\ad$)  be the unique positive root (resp. coroot) determined by $\Bd$. Define
\[ e_{\ad}:=\left(\begin{matrix}
0&1\\0&0
\end{matrix}\right) \in \dd{\gg}_{\ad}\quad \text{ and }\quad e_{-\ad}:=\left(\begin{matrix}
0&0\\1&0
\end{matrix}\right) \in \dd{\gg}_{-\ad}.\]
We have 
\[ e^T(h)= h\a +e_{\ad}.\]

The $T$-equivariant quantum cohomology $QH_T^{\bl}(\PP^1_{\CC})$ is generated by the equivariant parameter $h$, the quantum parameter $q$ and the unique Schubert class $\s$ which is not equal to 1, subject to
\[ \s^2=2h\s +q.\]
(The presence of the factor ``2'' in the above expression is due to the fact that we are using the maximal torus of $\mathbf{SL}_{2}$ instead of $\mathbf{PGL}_{2}$.) In other words,
\[ QH_T^{\bl}(\PP^1_{\CC})\simeq \ZZ[h,q,\s]/\langle \s^2-2h\s -q\rangle \simeq \ZZ[h,\s]. \]

Observe that the condition for the definition of $\YYY$ is always satisfied, and hence 
\[ \YYY=\GB\times\spec\R\simeq \PP^1_{\ZZ}\times \mathbb{A}^1_{\ZZ}.\]
The projective line $\PP^1_{\ZZ}$ is covered by two open subschemes 
\[ \Ud(e):=\Ud\Bmd/\Bmd \simeq \Ud\quad\text{ and }\quad \Ud(w_0):=\Ud_-\dot{w}_0\Bmd/\Bmd\simeq \Ud_-\]
where $w_0\in W\simeq S_2$ is the longest element. Then we have
\[ \YY\simeq \Ud(w_0)\times\mathbb{A}^1_{\ZZ}\quad\text{ and }\quad\YYs\simeq \left(\Ud(w_0)\cap\Ud(e)\right)\times\mathbb{A}^1_{\ZZ},\]
and hence 
\[  \OO(\YY)\simeq\ZZ[h,y_-]\quad\text{ and }\quad\OO(\YYs)\simeq\OO(\YY)[y_-^{-1}]\simeq\ZZ[h,y_-^{\pm 1}]\]
where $y_-:=\chi_-(u_-)$ is the composition of the $\Ud(w_0)$-coordinate $u_-$ of $\YY$ and the unique isomorphism $\chi_-:\Ud_-\xrightarrow{\sim}\ga$ satisfying $\lie(\chi_-)(e_{-\ad})=1$. The homomorphism
\[ 
\begin{array}{ccccc}
\Phi& :&\OO(\YY) &\ra & QH_T^{\bl}(\PP^1_{\CC})\\ [.5em]
& & y_- & \mapsto & -\s
\end{array}\]
is obviously an isomorphism of $\R$-algebras.

We also need to show that $\Phi$ is $\ZZ[q]$-linear. Let $e_{\ad}^*\in (\dd{\gg})^*$ be the unique element satisfying $\langle e_{\ad}^*,e_{\ad}\rangle = 1$ and $e_{\ad}^*|_{\dd{\bb}_-}=0$. Define
\[ \dot{w}_0:= \left(\begin{matrix}
0&1\\-1&0
\end{matrix}\right)\] 
considered as a $\ZZ$-point of $N_{\Gd}(\Td)$. It represents $w_0$ in the Weyl group. We have 
\begin{align*}
\ol{q} &= \langle e_{\ad}^*, (u_-\dot{w}_0)^{-1}\cdot (h\a+e_{\ad})\rangle\\
& = \langle e_{\ad}^*, \dot{w}_0^{-1}\cdot [-(y_-^2+2hy_-) e_{-\ad} + (y_-+h)\a+e_{\ad}]\rangle\\
& = y_-^2+2hy_-.
\end{align*}
It follows that $\Phi(\ol{q}) = \s^2-2h\s =q$ as desired.

Next we look at $\Phi^*$. Define 
\[ \Phi^*:=\Phi[y_-^{-1}]:\OO(\YYs)\xrightarrow{\sim} QH_T^{\bl}(\PP^1_{\CC})[\s^{-1}].\]
Observe that $\s^{-1}=q^{-1}(\s-2h)$, and hence 
\[ QH_T^{\bl}(\PP^1_{\CC})[\s^{-1}]\subseteq QH_T^{\bl}(\PP^1_{\CC})[q^{-1}].\]
But this inclusion is proper unless we set $h=0$. In fact, if $\kkk$ is an integral domain such that $2\ne 0$, then $\kkk[h,\s^{\pm1}]$ and $\kkk[h,\s,(\s^2-2h\s)^{-1}]$ are not isomorphic as $\kkk$-algebras. Indeed, after inverting 2, they are isomorphic to $\kkk_1[X,Y^{\pm 1}]$ and $\kkk_1[X^{\pm 1},Y^{\pm 1}]$ respectively, where $\kkk_1:=\kkk[2^{-1}]$. But 
\[\kkk_1[X,Y^{\pm 1}]^{\times}/\kkk_1^{\times}\simeq\ZZ\quad\text{ while }\quad\kkk_1[X^{\pm 1},Y^{\pm 1}]^{\times}/\kkk_1^{\times}\simeq\ZZ\oplus\ZZ. \]

We now look at $\Phi_{loc}$. Let $t$ and $u_+$ be the $\Td$- and $\Ud$-coordinates of $\Bd\simeq\Td\times\Ud$ respectively. Then 
\[ (tu_+)\cdot e^T(h) = h\a +\ad(t)(1-2h\chi_+(u_+))e_{\ad}\]
where $\chi_+:\Ud\xrightarrow{\sim} \ga$ is the unique isomorphism satisfying $\lie(\chi_+)(e_{\ad})=1$. Put $z:=\ad(t)$ and $y_+:=\chi_+(u_+)$. It follows that 
\[ \OO(\Bed)\simeq \ZZ[h,y_+,z^{\pm 1}]/\langle z(1-2hy_+)-1\rangle\simeq \ZZ[h,y_+,(1-2hy_+)^{-1}].\]
Observe that 
\[ t\chi_+^{-1}(y) = \chi_-^{-1}(\ad(t)^{-1}y^{-1})\dot{w}_0^{-1}(t^{-1}\a(-y^{-1}))\chi_-^{-1}(y^{-1})\]
for any $y\in \gm\subset\ga$ and $t\in \Td$. It follows that 
\[ \OO(\Bed\times_{\Gd}\Ud_-\dot{w}_0^{-1}\Bmd)\simeq\ZZ[h,y_+^{\pm 1}, (1-2hy_+)^{-1}] \]
and $\ol{q}_{loc}=z^{-1}y_+^{-2}=y_+^{-2}-2hy_+^{-1}$. It is straightforward to show that the homomorphism
\[ 
\begin{array}{ccccc}
\Phi_{loc}& :&\OO(\Bed\times_{\Gd}\Ud_-\dot{w}_0^{-1}\Bmd) &\ra & QH_T^{\bl}(\PP^1_{\CC})[q^{-1}]\\ [.5em]
& & y_+ & \mapsto & \s^{-1}=q^{-1}(\s-2h)
\end{array}\]
is an isomorphism of $\R$-algebras. Finally, 
\[ \Phi_{loc}( \ol{q}_{loc}) = \Phi_{loc}(y_+^{-2}-2hy_+^{-1}) = \s^2-2h\s =q. \]
Therefore, $\Phi_{loc}$ is $\ZZ[q]$-linear.

\begin{remark}
The isomorphisms $\Phi$ and $\Phi_{loc}$ are in fact intertwined by the homomorphism $\OO(\YYY_P)\ra \OO(\Bed\times_{\Gd}\Ud_-\dot{w}_0^{-1}\Bmd)$ induced by the morphism $\Bd\times\spec\R\ra \GB\times\spec\R$ defined by $(b,h)\mapsto (b^{-1}\Bmd,h)$.
\end{remark}
\subsection{Organization} In Section \ref{2}, we establish notation and recall facts from equivariant quantum cohomology, affine Grassmannian, geometric Satake equivalence and scheme of Borel subgroups over $\spec\ZZ$. In Section \ref{3}, we recall Peterson-Lam-Shimozono's theorem and compute the map with input different from affine Schubert classes, based on our recent new proof of the theorem. In Section \ref{4}, we recall Yun-Zhu's theorem and a result of Baumann, Kamnitzer and Knuston which describes the isomorphism in terms of Mirkovi\'c-Vilonen cycles. In Section \ref{5}, we introduce the Peterson schemes and verify several properties. In Section \ref{6}, we prove our main results. In Appendix \ref{appA}, we prove a number of lemmas which are used in the preceding sections but not proved immediately after stated. In Appendix \ref{app}, we settle sign issues which arise when we identify specific regular functions on the Peterson schemes with the quantum parameters. 


\section{Preliminaries}\label{2}
\subsection{Notation}\label{2a}
Let $G$ be a complex reductive group. Fix a maximal torus $T$ and a Borel subgroup $B$ containing $T$. Let $B_-$ be the unique Borel subgroup such that $B\cap B_-=T$. Denote by $R$ the set of roots and by $R^+$ the set of positive roots determined by $B$. The simple roots are denoted by $\a_1,\ldots,\a_r$. For any $\a\in R$, denote by $\ad$ the corresponding coroot. Let $W$ be the Weyl group and $w_0$ the longest element of $W$. Let $\oc$ (resp. $\co$) be the lattice of the characters (resp. cocharacters) of $T$. Let $\cop$ be the subset of $\co$ consisting of elements which are dominant. Let $\Q$ be the coroot lattice. In general, $\Q\subseteq\co$, and the equality holds precisely when $G$ is simply-connected. Define
\[ 2\rho:=\sum_{\a\in R^+}\a\in\oc\quad\text{ and }\quad 2\rhod:=\sum_{\a\in R^+}\ad\in\co.\]

Fix a parabolic subgroup $P$ of $G$ containing $B$. Let $L$ be its Levi subgroup. Denote by $R_P\subseteq R$ (resp. $R^+_P\subseteq R^+$) the set of (resp. positive) roots of $L$. WLOG, we assume 
\[R_P\cap\{\a_1,\ldots,\a_r\}=\{\a_{k+1},\ldots,\a_r\}\]
for some integer $0\leqslant k\leqslant r$. Let $W_P\subseteq W$ be the Weyl group of $L$ and $w_P\in W_P$ the longest element of $W_P$. Define $\Q_P$ to be the sublattice of $\Q$ generated by $\ad$ with $\a\in R_P$.
\subsection{Equivariant quantum cohomology of flag variety}\label{2bc}
Assume $G$ is simply-connected. Let $P$ be the standard parabolic subgroup of $G$ fixed in Section \ref{2a}. Denote by $W^P$ the set of minimal length coset representatives in $W/W_P$. For each $v\in W^P$, define 
\begin{align*}
 \s_v &:=\pd\bk{\ol{B_-\dot{v}P/P}} \in H_T^{2\ell(v)}(G/P)\\
  \s^v & :=\pd\bk{\ol{B\dot{v}P/P}}\in H_T^{\dim_{\RR}(G/P)-2\ell(v)}(G/P).
\end{align*}
The families $\{\s_v\}_{v\in W^P}$ and $\{\s^v\}_{v\in W^P}$ are dual bases of $H_T^{\bl}(G/P)$ with respect to $\int_{G/P}-\cup -$.

Let $\ec\subseteq \pi_2(G/P)$ be the semi-group of effective curve classes in $G/P$ and $\ZZ[\ec]$ the associated group ring. Since $\ec$ is identified with $\sum_{i=1}^k\ZZ_{\geqslant 0}\cdot [\ad_i]\subseteq \Q/\Q_P$ under the canonical isomorphism $\pi_2(G/P)\simeq \Q/\Q_P$, the ring $\ZZ[\ec]$ is isomorphic to the polynomial ring $\ZZ[q_1,\ldots,q_k]$ where $q_i:=q^{[\ad_i]}$. Define the $T$-equivariant quantum cohomology of $G/P$
\[ QH_T^{\bl}(G/P):= H_T^{\bl}(G/P)\otimes\ZZ[\ec].\] 
We grade $QH_T^{\bl}(G/P)$ by declaring each $q_i$ to have degree $2\sum_{\a\in R^+\setminus R^+_P}\a(\ad_i)$. The $T$-equivariant quantum cup product $\star$ is a deformation of the $T$-equivariant cup product, defined by
\[\s_u\star \s_v := \sum_{w\in W^P}\sum_{\beta\in\ec} q^{\b} \left(\int_{\MM_{0,3}(G/P,\beta)} \ev_1^*\s_u\cup \ev_2^*\s_v\cup  \ev_3^*\s^w \right) \s_w,\]  
where
\begin{enumerate}[(i)]
\item $\MM_{0,3}(G/P,\beta)$ is the moduli of genus-zero stable maps to $G/P$ of degree $\beta$ which have three marked points in the domain curve and 
\[\ev_1,\ev_2,\ev_3:\MM_{0,3}(G/P,\beta)\ra G/P\]
are the evaluation morphisms at these marked points; and

\item the integral $\int_{\MM_{0,3}(G/P,\beta)}$ is the $T$-equivariant integral.
\end{enumerate}

\noindent It is well-known (e.g. \cite{BM, FP}) that 
\begin{enumerate}
\item $\MM_{0,3}(G/P,\beta)$ is a smooth projective Deligne-Mumford stack of expected dimension. In particular, the above integral is well-defined.

\item $(QH_T^{\bl}(G/P),\star)$ is a graded commutative $\R$-algebra. 
\end{enumerate}
\subsection{Affine Grassmannian}\label{2b}
Let $\ag:=\ag_G:=G(\KK)/G(\OO)$ be the affine Grassmannian of $G$, where $\KK:=\CC((z))$ and $\OO:=\CC[[z]]$. Every $\ll\in\co$ is naturally an element of $G(\KK)$ and so gives rise to a point $t^{\ll}$ of $\ag$. The set $\{t^{\ll}\}_{\ll\in\co}$ is the set of $T$-fixed points of $\ag$. Suppose $\ll$ is dominant. Define $\agl:=G(\OO)\cdot t^{\ll}$ which is an affine bundle on $G\cdot t^{\ll}$. Its closure $\ol{\agl}$ is called a spherical affine Schubert variety. It is known that $\ol{\agl}=\bigcup_{\cop\ni\mu\leqslant\ll}\ag^{\mu}$ where $\mu\leqslant\ll$ if and only if $\ll-\mu$ is a sum of positive coroots. Because of this equality, we will also denote the closure by $\agll$.

The Cartan decomposition tells us $\ag=\bigcup_{\ll\in\cop}\agl$. We can further decompose each $\agl$ into smaller strata. Define $\BB:=\ev^{-1}_{z=0}(B_-)$ where $\ev_{z=0}:G(\OO)\ra G$ is the evaluation map at $z=0$. For any $\mu\in\co$ not necessarily dominant, the orbit $\BB\cdot t^{\mu}$ is isomorphic to an affine space and called an affine Schubert cell. Its closure $\ol{\BB\cdot t^{\mu}}$ is called an affine Schubert variety. We have $\agl=\bigcup_{\mu\in W\cdot \ll}\BB\cdot t^{\mu}$.

Let $K$ be a maximal compact subgroup of $G$ such that $T\cap K$ is a maximal torus of $K$. It is known that $\ag$ is homeomorphic, as $T\cap K$-spaces, to the polynomial based loop group $\O K$. The group structure on $\O K$ thus induces an $\R$-algebra structure on $H_{\bl}^{T\cap K}(\O K)\simeq \Hag$. This product is called the Pontryagin product and denoted by $\bulletsmall$. For any $\mu_1,\mu_2\in \co$, we have
\begin{equation}\nonumber \label{2beq1}
[t^{\mu_1}]\bulletsmall [t^{\mu_2}] = [t^{\mu_1+\mu_2}].
\end{equation} 
\subsection{Geometric Satake equivalence}\label{2c}
We give a brief review on the geometric Satake equivalence proved by Mirkovi\'c and Vilonen \cite{MV}. See also the work \cite{BD, Ginzburg, Lu}. Let $\per$ be the category of $G(\OO)$-equivariant perverse sheaves on $\ag$ with finite dimensional support. The coefficient ring here is $\ZZ$. It is an abelian category equipped with a tensor structure given by the convolution product and a fiber functor given by the hypercohomology $H^{\bl}(-)$. There exists a functorial isomorphism
\begin{equation}\label{4aweight}
H^{\bl}(\FF)\simeq \bigoplus_{\mu\in\co} H^{2\rho(\mu)}_{S_{\mu}^-}(\FF)
\end{equation}
for any object $\FF$, where 
\[  \MV:= \left\{ y\in\ag\left|~\lim_{s\to\infty}2\rhod(s)\cdot y= t^{\mu} \right.\right\}.\]
Using the information about $\per$, Mirkovi\'c and Vilonen constructed a group scheme $\widetilde{G}$ over $\spec\ZZ$ which contains the dual split torus $\Td :=\spec\ZZ[\co]$ as a maximal torus and satisfies
\begin{equation}
\per\simeq \rep(\widetilde{G})
\end{equation}
as tensor categories, where $\rep(\widetilde{G})$ is the category of representations of $\widetilde{G}$ which are finitely generated over $\ZZ$. Moreover, $H^{\bl}(-)$ (resp. $H_{S^-_{\mu}}^{2\rho(\mu)}(-)$) corresponds to the forgetful functor (resp. $\mu$-weight functor). They showed that $\widetilde{G}$ is split reductive and the root datum associated to $(\widetilde{G},\Td)$ is dual to that associated to $(G,T)$, and hence $\widetilde{G}\simeq \Gd$, the Langlands dual group scheme of $G$. 

We will consider a particular class of objects of $\per$ called costandard sheaves. For any $\ll\in\cop$, define
\begin{equation}\label{costandardsheafdef}
\JJ_{\ll}:=\PH((j_{\ll})_*\ZZ[2\rho(\ll)])
\end{equation}
where $\PH$ is the zeroth perverse cohomology functor and $j_{\ll}:\agl\hookrightarrow \agll$ is the inclusion map. 
\begin{lemma}\label{2clemma2} (\cite[Proposition 3.10]{MV}) For any $\lambda\in\cop$ and $\mu\in\co$, there exists an isomorphism between $H^{2\rho(\mu)}_{\MV}(\JJ_{\ll})$ and the free abelian group generated by the irreducible components of $\MV\cap \agl$. \hfill$\square$
\end{lemma}

\subsection{Scheme of Borel subgroups}\label{2d} The main reference for this subsection is \cite{BC}. Let $(\Gd,\Td)$ be as in Section \ref{2c}. By \cite[Theorem 4.1.7 \& Corollary 5.2.7]{BC}, there exist unique Borel subgroups $\Bd$ and $\Bd_-$ of $\Gd$ such that $\Td\subset \Bd$, $\Td\subset \Bd_-$,
\[ \mathfrak{b}^{\vee}:=\lie(\Bd) = \lie(\Td)\oplus \bigoplus_{\a\in R^+}\gg^{\vee}_{\ad}\quad\text{ and }\quad \mathfrak{b}^{\vee}_-:=\lie(\Bd_-) = \lie(\Td)\oplus \bigoplus_{\a\in -R^+}\gg^{\vee}_{\ad}.\]
Let $\Ud$ (resp. $\Ud_-$) denote the unipotent radical of $\Bd$ (resp. $\Bmd$). By \cite[Corollary 5.2.8]{BC}, the \'etale quotient sheaf $\GB$, i.e. the \'etale sheafification of the presheaf $R\mapsto \Gd(R)/\Bmd(R)$, is represented by a smooth projective scheme over $\spec\ZZ$. 

Let $\ll\in\cop$ and $\JJ_{\ll}$ be the costandard sheaf defined in \eqref{costandardsheafdef}. Define 
\[S(\ll):=H^{\bl}(\JJ_{\ll})\] which is a $\Gd$-module by the geometric Satake equivalence (Section \ref{2c}). Its $\Td$-weights are $\leqslant\ll$. For any $\mu\in\co$, define $S(\ll)_{\mu}:=H^{2\rho(\mu)}_{\MV}(\JJ_{\ll})$, the $\mu$-weight space of $S(\ll)$. 

\begin{definition}\label{2ddef}
Let $\ll\in\cop$.
\begin{enumerate}
\item For any $w\in W$, define $v_w\in S(\ll)_{w(\ll)}\simeq\ZZ$ to be the generator corresponding to the unique irreducible component of $S^-_{w(\ll)}\cap \agl$ via the isomorphism in Lemma \ref{2clemma2}. 

\item Define $v_{\ll}^*\in S(\ll)^*$ to be the unique element vanishing on $S(\ll)_{\mu} $ for any $\mu\ne\ll$ and sending $v_e\in S(\ll)_{\ll}$ defined in (1) to 1.
\end{enumerate}
\end{definition}

Define $M_{\ll}$ to be the rank-one $\Bmd$-module satisfying
\[ \res^{\Bmd}_{\Ud_-}(M_{\ll})\text{ is trivial}\quad\text{and}\quad\res^{\Bmd}_{\Td}(M_{\ll})\text{ has weight }\ll. \]
Define 
\begin{equation}\label{2dlinebundle}
\LL(\ll):=\Gd\times_{\Bmd}M_{\ll}
\end{equation}
to be the descent, along the quotient morphism $\Gd\ra\GB$, of the $\Bmd$-equivariant line bundle $\OO_{\Gd}\otimes M_{\ll}$. It is a line bundle on $\GB$. The element $v_{\ll}^*$ from Definition \ref{2ddef} can be regarded as a morphism $S(\ll)\ra M_{\ll}$ of $\Bmd$-modules. Hence it gives rise to a morphism
\begin{equation}\label{2deq1}
\OO_{\GB}\otimes S(\ll)\simeq \Gd\times_{\Bmd} S(\ll)\ra \LL(\ll)
\end{equation}
of $\OO_{\GB}$-modules, where the trivialization of $\Gd\times_{\Bmd} S(\ll)$ is induced by $(g,v)\mapsto g^{-1}\cdot v$.
\begin{lemma}\label{2dlemma1}
For any $\ll\in\cop$, the linear map
\begin{equation}\label{2deq15}
 S(\ll)\simeq H^0(\GB;\Gd\times_{\Bmd} S(\ll))\ra H^0(\GB;\LL(\ll))
\end{equation}
induced by \eqref{2deq1} is bijective.
\end{lemma}
\begin{proof}
The following arguments can be found in \cite[Section 13]{MV}. Notice that 
\[H^0(\GB;\LL(\ll))\simeq \ind_{\Bmd}^{\Gd}(M_{\ll})\] 
as $\Gd$-modules. The latter $\Gd$-module is an object of $\rep^{\leqslant\ll}(\Gd)$, the full subcategory of $\rep(\Gd)$ consisting of objects whose $\Td$-weights are all $\leqslant\ll$, and its $\ll$-weight space is isomorphic to $\ZZ$. Moreover, for any $V\in\rep^{\leqslant\ll}(\Gd)$, the Frobenius reciprocity gives
\begin{equation}\label{2deq2}
\home_{\rep(\Gd)}\left(V,\ind_{\Bmd}^{\Gd}(M_{\ll})\right)\simeq \home_{\rep(\Bmd)}\left(\res_{\Bmd}^{\Gd}(V),M_{\ll}\right)\simeq \home_{\ZZ}(V_{\ll},\ZZ).
\end{equation}
See e.g. \cite{Jantzen} for proofs of the above claims. Since $S(\ll)$ satisfies the same universal property, we have 
\[S(\ll)\simeq \ind_{\Bmd}^{\Gd}(M_{\ll})\]
as $\Gd$-modules. Now take $V=S(\ll)$. The proof is complete if we can show that the element of the RHS of \eqref{2deq2} corresponding to our map \eqref{2deq15} is an isomorphism. This follows from the assumption that $v_{\ll}^*$ sends a generator of $S(\ll)_{\ll}\simeq\ZZ$ to $1$. (Recall \eqref{2deq2} is defined by evaluating the sections of $\LL(\ll)$ at $e\Bmd$.)
\end{proof}

Consider the Weyl group of $(\Gd,\Td)$ which is nothing but $W$ viewed as a constant group scheme over $\spec\ZZ$. For any $w\in W$, choose a representative $\dot{w}\in N_{\Gd}(\Td)(\ZZ)$ of $w$. (In later sections, we will consider specific choices for some $w$. But for now, the choice is not important.) Define $\Ud(w):=\dot{w}\Ud\Bmd/\Bmd$. It is an open Schubert cell of $\GB$ with respect to the Borel subgroup $\Bd_w:=\dot{w}\Bd\dot{w}^{-1}$. See \cite[Theorem 4.1.7]{BC}.  

For any $v\in S(\ll)$, denote by $s_{\ll,v}\in H^0(\GB;\LL(\ll))$ its image under \eqref{2deq15}. An important fact about $s_{\ll,v}$ is that the pull-back of $s_{\ll,v}$ via the quotient morphism $\Gd\ra\GB$ is equal to the regular function $g\mapsto \langle v_{\ll}^*,g^{-1}\cdot v\rangle$ (after identifying $\OO(\Gd)\otimes M_{\ll}$ with $\OO(\Gd)$).

\begin{lemma}\label{2dlemma3}
For any $\ll\in\cop$ and $w\in W$, the section $s_{\ll,v_w}$ is non-vanishing on $\Ud(w)$. If $\ll$ is regular dominant, $s_{\ll,v_w}$ is non-vanishing precisely on $\Ud(w)$. ($v_w$ is defined in Definition \ref{2ddef}.)
\end{lemma}
\begin{proof}
It suffices to verify the assertions for the geometric points so we assume that the base scheme is $\spec \kkk$ for some algebraically closed field $\kkk$. In this case, we have the Bruhat decomposition for $\GB$. Since $ \{s_{\ll,v_w}\ne 0\}$ is $\Bd_w$-invariant, we only need to show $s_{\ll,v_w}(\dot{w}\Bmd)\ne 0$, and if $\ll$ is regular dominant, then $s_{\ll,v_w}(\dot{v}\Bmd)=0$ for any $v\in W\sm\{w\}$. This follows from the equality
\[s_{\ll,v_w}(\dot{v}\Bmd)=\langle v_{\ll}^*,\dot{v}^{-1}\cdot v_w \rangle\]
(up to a non-zero multiple) and the fact that $\dot{v}^{-1}\cdot v_w $ is a generator of $S(\ll)_{v^{-1}w(\ll)}$. 
\end{proof}

\begin{lemma}\label{2dlemma2}
There exists $\ll_0\in\ZZ_{>0}\cdot 2\rhod$ such that $\LL(\ll_0)$ is very ample.
\end{lemma}
\begin{proof}
Put $X:=\GB$ and let $f:\Gd\ra X$ be the quotient morphism. By \cite[Remark 2.3.7 \& Theorem 5.2.11]{BC}, the determinant $\det(\TTT_X)$ of the tangent bundle $\TTT_X:=\TTT_{X/\spec\ZZ}$ is ample. Hence it suffices to show $\det(\TTT_X)\simeq \LL(2\rhod)$. By descent, it amounts to showing
\[ f^*\det(\TTT_X)\simeq \OO_{\Gd}\otimes M_{2\rhod}\]
as $\Bmd$-equivariant sheaves.

Since $f$ is smooth, we have a short exact sequence
\begin{equation}\label{2deeq0}
0\ra \TTT_{\Gd/X}\ra \TTT_{\Gd}\ra f^*\TTT_X\ra 0.
\end{equation}
Let us determine the second arrow. Recall it is induced by the canonical morphism
\begin{equation}\label{2deeq1}
\Gd\times_X\Gd\ra \Gd\times\Gd.
\end{equation}
Observe that $(g_1,g_2)\mapsto (g_1,g^{-1}_1g_2)$ gives rise to isomorphisms
\[ \Gd\times_X\Gd\simeq \Gd\times\Bmd\quad\text{ and }\quad \Gd\times\Gd\simeq  \Gd\times\Gd \]
under which \eqref{2deeq1} is identified with $\id_{\Gd}\times\iota$ where $\iota:\Bmd\hookrightarrow \Gd$ is the inclusion, and the diagonal morphisms become $g\mapsto (g,e)$. It follows that the second arrow in \eqref{2deeq0} is the monomorphism $\OO_{\Gd}\otimes\dd{\bb}_-\ra \OO_{\Gd}\otimes\dd{\gg}$ induced by the inclusion $\dd{\bb}_-\hookrightarrow\dd{\gg}$, and hence
\[  f^*\det(\TTT_X)\simeq \det(f^*\TTT_X)\simeq \det(\OO_{\Gd}\otimes \dd{\gg}/\dd{\bb}_-)\simeq \OO_{\Gd}\otimes \det( \dd{\gg}/\dd{\bb}_-).\]

The result will follow if we can show that $\det( \dd{\gg}/\dd{\bb}_-)\simeq M_{2\rhod}$ as $\Bmd$-modules. It is clear that they are isomorphic as $\Td$-modules. By definition, $M_{2\rhod}$ is a trivial $\Ud_-$-module. The same holds for $\det( \dd{\gg}/\dd{\bb}_-)$ as well because every graded piece associated to the $2\rho$-filtration on $\dd{\gg}/\dd{\bb}_-$ is a trivial $\Ud_-$-module. The proof is complete.
\end{proof}

\begin{proposition} \label{2dcor}
Let $\ll_0$ be the element from Lemma \ref{2dlemma2}. For any $w\in W$, the ring $\OO(\Ud(w))$ is generated by regular functions $s_{\ll_0,v}/s_{\ll_0,v_w}$ with $v\in S(\ll_0)$. 
\end{proposition}
\begin{proof}
This follows from Lemma \ref{2dlemma1}, Lemma \ref{2dlemma3} and Lemma \ref{2dlemma2}.
\end{proof}
\section{Peterson-Lam-Shimozono's theorem}\label{3}
\subsection{Statement}\label{3a}
Assume $G$ is simple and simply-connected. Denote by $W_{af}$ the affine Weyl group and by $W_{af}^-$ the set of minimal length coset representatives in the quotient $W_{af}/W$. Elements of $W_{af}$ are denoted by $wt_{\ll}$ where $w\in W$ and $t_{\ll}$ ($\ll\in\Q$) are translations $x\mapsto x+\ll$. Notice that the map $\wl\mapsto w(\ll)$ defines a bijection $W_{af}^-\simeq \Q$. For any $\wl\in W_{af}^-$, define the affine Schubert class
\[ \xi_{\wl} := \bk{\ol{\BB\cdot t^{w(\ll)}}}\in H^T_{2\ell(\wl)}(\ag). \] 
It is known that $\{\xi_{\wl}\}_{\wl\in W_{af}^-}$ forms an $\R$-basis of $\Hag$. Define $(W^P)_{af}$ to be the set of $\wl\in W_{af}$ such that
\[\left\{
\begin{array}{rcl}
\a\in R_P^+\cap w^{-1}R^+&\Longrightarrow& \a(\ll)=0\\ [.5em]
\a\in R_P^+\cap (-w^{-1}R^+) &\Longrightarrow& \a(\ll)=-1
\end{array}
\right. .\]
For any $w\in W$, denote by $\widetilde{w}$ the unique element of $W^P$ satisfying $wW_P=\widetilde{w}W_P$.
\begin{theorem}(\cite{me, LS, Peter})\label{3aPLS}
The $\R$-linear map
\[\Phi_{PLS}: H_{-\bl}^T(\ag)\ra QH^{\bl}_T(G/P)[\ecc]\]
defined by
\[\Phi_{PLS}(\xi_{\wl}):=\left\{
\begin{array}{cc}
q^{[\ll]}\s_{\widetilde{w}}& \wl\in (W^P)_{af}\\ [.5em]
0& \text{otherwise}
\end{array}
\right.
\]
is a graded homomorphism of $\R$-algebras. \hfill$\square$
\end{theorem}
\subsection{As Savelyev-Seidel's homomorphism}\label{3b}
We briefly recall our recent new proof of Theorem \ref{3aPLS} \cite{me} which interprets $\Phi_{PLS}$ as Savelyev-Seidel's homomorphism \cite{Savelyev1, Savelyev2, Savelyev3, Seidel}. This will be crucial for the construction of the isomorphisms stated in Theorem \ref{1athm} and Theorem \ref{1athmb}.

By Beauville-Laszlo's theorem \cite{BL}, the affine Grassmannian $\ag$ of $G$ parametrizes isomorphism classes of $G$-torsors over $\PP^1$ with trivializations over $\PP^1\sm 0$. For any morphism $f:\G\ra \ag$, denote by $\fib{f}$ the associated object which is a $G$-torsor over $\PP^1\times \G$ with a trivialization over $(\PP^1\sm 0)\times \G$. For any $G$-variety $X$, define $\fib{f}(X):=\fib{f}\times_G X$. We will mainly consider the case $X=G/P$. Define $D_{f,\infty}:=\pi_f^{-1}(\infty\times\G)\subset \ffib{f}$ where $\pi_f:\ffib{f}\ra \PP^1\times\G$ is the projection. Let $\iota_{f,\infty}:D_{f,\infty}\hookrightarrow \ffib{f}$ be the inclusion. The above trivialization induces an isomorphism $D_{f,\infty}\simeq \G\times G/P$. Let $\rho\in (\Q/\Q_P)^*$. By descent, there exists a line bundle $\LL_{\rho}$ over $\ffib{f}$ whose restriction to each fiber of $\pi_f$ is isomorphic to the line bundle $G\times_P\CC_{-\rho}$ over $G/P$. 

Now assume $\G$ is a smooth projective variety. Then $\ffib{f}$ is a smooth projective variety by \cite[Lemma 3.2]{me}.
\begin{definition}\label{3bdef} Let $\eta\in \Q/\Q_P$.
\begin{enumerate}
\item Define 
\[ \MM(f,\eta):=\bigcup_{\b}\MM_{0,1}(\ffib{f},\b)\times_{(\ev_1,\iota_{f,\infty})}D_{f,\infty}\]
where $\b$ runs over all curve classes in $\ffib{f}$ such that $(\pi_{f})_*\b=[\PP^1\times\pt]$ and 
\[\langle c_1(\LL_{\rho}),\b\rangle=\rho(\eta)\]
for any $\rho\in (\Q/\Q_P)^*$.

\item Define
\[ \ev_{f,\eta}: \MM(f,\eta)\ra G/P\]
to be the composition of the morphisms
\[ \MM(f,\eta)\xrightarrow{\ev_1} D_{f,\infty}\simeq \G\times G/P\ra G/P\]
where the last arrow is the projection.

\item Define 
\[ \Pi_{f,\eta}:\MM(f,\eta)\ra \G\]
to be the morphism sending every stable map $u$ to the point in $\G$ over which $u$ lies, i.e. $\im(\pi_f\circ u)=\PP^1\times\Pi_{f,\eta}([u])$.
\end{enumerate}
\end{definition}

\begin{lemma}\label{3bdimlemma} (\cite[Lemma 3.7]{me}) The virtual dimension of $\MM(f,\eta)$ is equal to $\dim\G+\dim G/P+\sum_{\a\in R^+\sm R^+_P}\a(\eta)$. \hfill $\square$
\end{lemma}

Let $H$ be a subgroup of $G$. Suppose $\G$ has an $H$-action such that $f$ is $H$-equivariant. By \cite[Lemma 3.5]{me}, $\ffib{f}$ has a natural $H$-action such that $\pi_f$ is $H$-equivariant, $D_{f,\infty}$ is $H$-invariant and the above isomorphism $D_{f,\infty}\simeq \G\times G/P$ is $H$-equivariant. It follows that each stack $\MM(f,\eta)$ has a natural $H$-action such that both $\ev_{f,\eta}$ and $\Pi_{f,\eta}$ are $H$-equivariant.

\begin{definition} Define Savelyev-Seidel's homomorphism
\[ \Phi_{SS}: \Hmag\ra \QHPi\]
by
\[ \Phi_{SS}(\xi_{\wl}):=\sum_{v\in W^P}\sum_{\eta\in\Q/\Q_P}q^{\eta}\left( \int_{[\MM(f_{\ag,\wl},\eta)]^{vir}}\ev_{f_{\ag,\wl},\eta}^*\s^v\right)\s_v\]
where $f_{\ag,\wl}:\G_{\wl}\ra\ag$ is a $T$-equivariant morphism which factors through a resolution $\G_{\wl}\ra\ol{\BB\cdot t^{w(\lambda)}}$, e.g. Bott-Samelson resolution.
\end{definition}

Theorem \ref{3aPLS} follows from
\begin{theorem}\label{3bthm} (\cite[Proposition 3.12 \& Theorem 4.9]{me}) 
\begin{enumerate}
\item $\Phi_{SS}$ is a graded homomorphism of $\R$-algebras.
\item $\Phi_{SS}=\Phi_{PLS}$.
\end{enumerate}
\hfill $\square$
\end{theorem}

A key step of the proof is to show that $\MM(f_{\ag,\wl},\eta)$ is smooth and of expected dimension for a choice of $f_{\ag,\wl}$. For any $\a\in R$ and $k\in\ZZ$, define the affine root group $U_{\a,k}:=\exp(z^k\gg_{\a})\subset\LL G$. 
\begin{proposition}\label{3bDM} (\cite[Proposition 4.5]{me}) Let $f:\G\ra\ag$ be a morphism. Suppose, for any $\a\in R$ and $k>0$, $\G$ has a $U_{\a,k}$-action such that $f$ is $U_{\a,k}$-equivariant. Then for any $\eta\in \Q/\Q_P$ the stack $\MM(f,\eta)$ is smooth and of expected dimension. \hfill$\square$
\end{proposition}

The condition in Proposition \ref{3bDM} is fulfilled by taking e.g. Bott-Samelson resolutions or functorial resolutions \cite[Proposition 3.9.1 \& Theorem 3.26]{equivresolution}.

\subsection{Input different from affine Schubert classes} 
Let $\lambda\in\Q$ be dominant. Take a $G$-equivariant morphism $r_{\lambda}:\G_{\ll}\ra \ag$ which factors through a resolution $\G_{\ll}\ra\agll$ and satisfies the condition in Proposition \ref{3bDM}. It follows that, by Proposition \ref{3bDM}, the stack $\MM(r_{\ll},\eta)$ is smooth and of expected dimension for any $\eta\in\Q/\Q_P$. Moreover, since $r_{\ll}$ is $G$-equivariant, $\ev_{r_{\ll},\eta}$ is $G$-equivariant and so transverse to any smooth subvariety of $G/P$.
\begin{lemma}\label{3bcharacterize}
For any $x\in H_{\bl}^T(\G_{\ll})$, we have
\begin{equation}\label{3beq1}
\Phi_{SS}((r_{\ll})_*x)=\sum_{v\in W^P}\sum_{\eta\in \Q/\Q_P} q^{\eta}\left( x\bullet (\Pi_{r_{\ll},\eta})_*\ev_{r_{\ll},\eta}^*\s^v\right) \s_v
\end{equation}
where $\bullet$ is the intersection product on $\G_{\ll}$.
\end{lemma}
\begin{proof}
This follows from the standard localization argument. See e.g. the proof of \cite[Proposition 3.12]{me}.
\end{proof}


\begin{lemma}\label{3bnoempty}
The stack $\MM(r_{\ll},\eta)$ is non-empty if and only if $\eta\in [w_0(\ll)]+\ec$.
\end{lemma}
\begin{proof}
It suffices to look at the $T$-invariant sections of $\ffib{t^{\mu}}$ with $t^{\mu}\in \agll$. By \cite[Lemma 4.2]{me}, these sections are in bijective correspondence with the $T$-fixed points of $G/P$, and for each $T$-fixed point $\dot{v}P\in G/P$, the degree of the corresponding section is equal to $[v^{-1}(\mu)]$. The rest is clear.
\end{proof}

\begin{proposition}\label{3bzero}
For any dominant $\ll\in\Q$, $\Phi_{PLS}$ vanishes on $H^T_{>2\rho(\ll-\wp(\ll))}(\agll)$.
\end{proposition}
\begin{proof}
Let $y\in H^T_{>2\rho(\ll-\wp(\ll))}(\agll)$. We may assume $y$ is an affine Schubert class $
\bk{~\ol{\BB\cdot t^{\mu}}~}$ for some $\mu\in W\cdot\ll$. (If not, replace $\ll$ by another $\ll'$ with $\ll'<\ll$.) Then $y= (r_{\ll})_*x$ for some $x$ so that we can apply Lemma \ref{3bcharacterize}. By Lemma \ref{3bdimlemma}, for any $v\in W^P$ and $\eta\in \Q/\Q_P$, the homology degree of $(\Pi_{r_{\ll},\eta})_*\ev_{r_{\ll},\eta}^*\s^v$ is greater than or equal to $4\rho(\ll)+2\sum_{\a\in R^+\sm R^+_P}\a(\eta)$ which is greater than or equal to $4\rho(\ll)+2\sum_{\a\in R^+\sm R^+_P}\a(w_0(\ll))$ if $\MM(r_{\ll},\eta)\ne\emptyset$, by Lemma \ref{3bnoempty}. The last expression is equal to $4\rho(\ll)-2\rho(\ll-\wp(\ll))$. It follows that
\[ \deg\left((\Pi_{r_{\ll},\eta})_*\ev_{r_{\ll},\eta}^*\s^v\right) + \deg x> 4\rho(\ll) = \dim_{\RR} \G_{\ll} ,\]
and hence the intersection product in \eqref{3beq1} is zero. The result now follows from Theorem \ref{3bthm}.
\end{proof}

\begin{lemma}\label{3blevi}
For any dominant $\ll\in\Q$, the restriction of $\Pi_{r_{\ll},[w_0(\ll)]}$ to $\ev_{r_{\ll},[w_0(\ll)]}^{-1}(\dot{e}P)$ is an isomorphism onto $r_{\ll}^{-1}(\agp)$ where $L$ is the Levi subgroup of $P$ and
\[\agp:=\ol{L(\OO)\cdot t^{\wp(\ll)}} \subseteq \agll.\] 
\end{lemma}
\begin{proof}
Observe that the structure group of the $G$-torsor $\fib{\agp}$ is reduced to $L$. In fact, the associated $L$-torsor is nothing but $\fibl{\agp}$. Hence
\[ \left. \ffib{\agll}\right|_{\PP^1\times \agp}\simeq \ffibl{\agp}.\]
The $L$-equivariant morphism $\{\dot{e}P\}\hookrightarrow G/P$ induces a morphism 
\begin{equation}\label{3beq2}
\PP^1\times \agp\simeq \fibl{\agp}(\{\dot{e}P\})\hookrightarrow \ffibl{\agp}.
\end{equation}
In other words, every point in $\agp$ gives rise to a section of $\ffib{\agll}$ lying over the same point. It is not difficult to see that the degrees of these sections are equal to $[w_0(\ll)]$. Thus, \eqref{3beq2} defines a closed immersion 
\begin{equation}\label{3beeq1}
r_{\ll}^{-1}(\agp)\hookrightarrow \ev^{-1}_{r_{\ll},[w_0(\ll)]}(\dot{e}P).
\end{equation}
By a dimension argument and the smoothness of $\ev^{-1}_{r_{\ll},[w_0(\ll)]}(\dot{e}P)$, the image of \eqref{3beeq1} is a connected component of $\ev^{-1}_{r_{\ll},[w_0(\ll)]}(\dot{e}P)$. 

It remains to show that $\ev^{-1}_{r_{\ll},[w_0(\ll)]}(\dot{e}P)$ has no other connected component. Every other component has a $T$-fixed point, and by Lemma \ref{3bnoempty} every $T$-fixed point is represented by a $T$-invariant section of $\ffib{t^{\mu}}$ for some $\mu\in \Q\cap\conv(W\cdot \ll)$. Since the marked point on this section (over $\infty$) is required to hit $\dot{e}P\in G/P$, its degree is equal to $[\mu]$. It follows that $[\mu]=[w_0(\ll)]$, and hence $\mu\in\Q\cap\conv(W_P\cdot w_0(\ll))$ or equivalently $t^{\mu}\in \agp$. Thus, this section, viewed as a point of the component of $\ev^{-1}_{r_{\ll},[w_0(\ll)]}(\dot{e}P)$ we are looking at, belongs to the component we have already determined, a contradiction.
\end{proof}

\begin{proposition}\label{3bone}
Let $\ll\in\Q$ be dominant. For any $x\in H^T_{2\rho(\ll-\wp(\ll))}(\G_{\ll})$, we have 
\[ \Phi_{PLS}((r_{\ll})_*x)=q^{[w_0(\ll)]} x\bullet \left[ r_{\ll}^{-1}\left(\agp\right)\right] .\]
\end{proposition}
\begin{proof}
By the degree argument in the proof of Proposition \ref{3bzero}, the summation \eqref{3beq1} contains only one possibly non-zero summand, namely the term corresponding to $v=e$ and $\eta=[w_0(\ll)]$. The assertion then follows from Theorem \ref{3bthm}, Lemma \ref{3bcharacterize} and Lemma \ref{3blevi}.
\end{proof}
\section{Yun-Zhu's theorem}\label{4}
\subsection{Statement}\label{4a}
Recall the notation introduced in Section \ref{2c}. Assume $G$ is simple.

Consider the $T$-equivariant hypercohomology functor $H^{\bl}_T(-)$ on $\per$. By \cite[Lemma 2.2 \& Lemma 2.4]{YZ}, $H^{\bl}_T(-)\simeq H^{\bl}(-)\otimes H_T^{\bl}(\pt)$ as tensor functors. Define a natural transformation $\s_{can}$ from the functor $H^{\bl}_T(-)\otimes_{\R} \Hmag$ to itself
\[ \s_{can}(v\otimes h):= \sum_i (h^i\cup v)\otimes (h_i\bulletsmall h)\] 
for any $v\in H^{\bl}_T(\FF)$ with $\FF\in \per$ and $h\in \Hmag$, where $\{h_i\}$ and $\{h^i\}$ are dual bases of $\Hmag$ and $H_T^{\bl}(\ag)$ respectively. By \cite[Lemma 3.1]{YZ}, $\s_{can}$ is a tensor automorphism, and hence it induces, via the Tannakian formalism, a homomorphism of $\R$-algebras
\[ \widetilde{\Phi}_{YZ}:\OO(\Gd\times \spec\R )\ra \Hmag\]
satisfying the property that for any $\FF\in \per$, $v\in H^{\bl}_T(\FF)$ and $v^*\in H^{\bl}_T(\FF)^*$, 
\begin{equation}\label{4aeq}
\widetilde{\Phi}_{YZ}(\vp(\FF,v,v^*)) = \sum_i\langle v^*, h^i\cup v\rangle h_i
\end{equation}
where $\vp(\FF,v,v^*)\in\OO(\Gd\times \spec H_T^{\bl}(\pt))$ is defined by
\[ \vp(\FF,v,v^*)(g,h):=\langle v^*, g\cdot v\rangle.\]


$\widetilde{\Phi}_{YZ}$ is graded if we grade $\OO(\Gd)$ using the conjugate action of $2\rho:\gm\ra \Td$. Moreover, by showing that $\s_{can}$ preserves the weight filtration induced by the $T$-equivariant analogue of the decomposition \eqref{4aweight}, $\widetilde{\Phi}_{YZ}$ factors through $\OO(\Bd\times\spec\R)$.

Consider an element $c\in H^2_T(\ag;\QQ)$ which is the unique primitive $T$-equivariant lift of the positive generator of $H^2(\ag;\ZZ)$. By the facts that $c$ is primitive and $\s_{can}$ commutes with the cup product $c\cup -$, $\widetilde{\Phi}_{YZ}$ further factors through the coordinate ring $\OO(\Bed)$ of the centralizer group scheme in $\Bd\times\spec\R$ of an element $e^T\in \dd{\gg}\otimes\R\otimes\QQ$ which is induced by the natural transformation
\[c\cup -:H^{\bl}_T(-)\ra H^{\bl+2}_T(-)\]
via the Tannakian formalism. By \cite[Proposition 5.6 \& Proposition 5.7]{YZ},
\begin{equation}\label{4aelement} 
e^T=e+f
\end{equation}
where 
\[ e:= \sum_{i=1}^r|\al_i|^2 e_{\ad_i}\quad\text{ and }\quad f:=2\frac{\langle-,-\rangle_{Kil}}{\langle\ad_0,\ad_0\rangle_{Kil}}.\]
Here,
\begin{enumerate}
\item $|\cdot|$ is a $W$-invariant norm on $\co_\RR$, normalized such that $|\al_i|^2=1$ if $\al_i$ is short;

\item $e_{\ad_i}$ is a generator of $\dd{\gg}_{\ad_i}\simeq\ZZ$; and

\item $\langle-,-\rangle_{Kil}$ is the Killing form regarded as an element of $\oc\otimes \oc\subset\dd{\gg}\otimes\R$.
\end{enumerate}

Observe that $f$, regarded as an element of $\home_{\QQ}(\co_\QQ,\oc_\QQ)$, sends each $\ad_i$ to $|\ad_i|^2 \a_i$. Suppose $G$ is simply-connected. Then $\co=\Q$. Since $\Q$ is generated by the simple coroots over $\ZZ$, we conclude that $f$ is defined over $\ZZ$.

\begin{theorem}(\cite[Theorem 6.1]{YZ}\label{4aYZ})
Assume $G$ is simple and simply-connected. The homomorphism 
\[ \Phi_{YZ}:\OO(\Bed)\ra \Hmag\]
induced by $\widetilde{\Phi}_{YZ}$ is bijective after inverting $\ell_G\in\ZZ$, the square of the ratio of the lengths of the long and short coroots of $G$. \hfill$\square$
\end{theorem}

\begin{remark} In \cite{BFM}, Bezrukavnikov, Finkelberg and Mirkovi\'c constructed the same isomorphism for the case where the coefficient ring is $\CC$. Their method does not rely on the geometric Satake equivalence. 
\end{remark}

\subsection{Description in terms of Mirkovi\'c-Vilonen cycles}\label{4b}
We describe Yun-Zhu's isomorphism $\Phi_{YZ}$ from Theorem \ref{4aYZ} in terms of Mirkovi\'c-Vilonen cycles (MV cycles) in $\ag$. When the coefficient ring is $\CC$, this was done in \cite[Section 10.2]{Acta}. The integral case is similar. For the convenience of the reader, we provide the details.

Let $\mu\in \co$. Recall
\[  \MV:= \left\{ y\in\ag\left|~\lim_{s\to\infty}2\rhod(s)\cdot y= t^{\mu}\right. \right\}.\]
It is a subset of $\ag$ whose dimension and codimension are both infinite. 
\begin{lemma}(\cite[Proposition 3.1]{MV}) \label{4bclosure}
$\ol{\MV}=\bigcup_{\nu\geqslant\mu}S_{\nu}^-$. \hfill $\square$
\end{lemma}

\begin{lemma}(\cite[Theorem 3.2]{MV})\label{4bdim}
Let $\ll\in\cop$ and $\mu\in\co$. Every irreducible component of $\MV\cap \agl$ has dimension $\rho(\ll-\mu)$.  \hfill $\square$
\end{lemma}

\begin{definition}\label{4bdef}
The closure of any irreducible component of $\MV\cap \agl$ is called an MV cycle of type $\ll$ and weight $\mu$. 
\end{definition}

The following lemma is well-known. We provide the proof for the convenience of the reader.
\begin{lemma}\label{4bequal}
The following two sets are equal:
\begin{enumerate}
\item The set of MV cycles of type $\ll$ and weight $\mu$.

\item The set of irreducible components of $\ol{\MV}\cap\agll$ which have dimension $\rho(\ll-\mu)$. 
\end{enumerate}
\end{lemma}
\begin{proof}
Denote by $A$ and $B$ the sets in (1) and (2) respectively. By $\agll=\bigcup_{\cop\ni \nu \leqslant\ll}\ag^{\nu}$, Lemma \ref{4bclosure} and Lemma \ref{4bdim}, the dimension of every irreducible component of $\ol{\MV}\cap\agll$ is less than or equal to $\rho(\ll-\mu)$. This shows $A\subseteq B$. Conversely, let $Z\in B$. Put $Z':=Z\cap \MV\cap\agl$. By the same reasoning, $Z'$ is a non-empty open dense subset of $Z$. It follows that $Z'$ is irreducible, and hence it is contained in an irreducible component $Z''$ of $\MV\cap\agl$. By Lemma \ref{4bdim} again, we have $Z=\ol{Z''}$, and hence $Z\in A$. 
\end{proof}

Let $\ll\in\cop$. Consider the costandard sheaf $\JJ_{\ll}$ defined in \eqref{costandardsheafdef}. Put $S(\ll):=H^{\bl}(\JJ_{\ll})$. Let $v_{\ll}^*\in S(\ll)^*$ be the element from Definition \ref{2ddef}. For any $v\in S(\ll)$, define $\tt(\JJ_{\ll},v,v_{\ll}^*)\in\OO(\Bed)$ by
\[ \tt(\JJ_{\ll},v,v_{\ll}^*)(b,h):=\langle v_{\ll}^*,b\cdot v\rangle. \]
\begin{proposition}\label{4bmv2}
(C.f. \cite[Lemma 10.5]{Acta}) Let $\mu\in\co$ and $v_Z\in S(\ll)_{\mu}$ be the element corresponding to an MV cycle $Z$ via the isomorphism from Lemma \ref{2clemma2}. We have
\[ \Phi_{YZ}(\tt(\JJ_{\ll},v_Z,v_{\ll}^*))=[Z]\in H^T_{2\rho(\ll-\mu)}(\agll;\ZZ). \]
\end{proposition}
\begin{proof}
For a space $X$, denote by $\DDD_X$ its dualizing complex. Let $\DDD$ denote the Verdier dual functor. By the perversity condition on $\DDD(\JJ_{\ll})$ and excision, we have
\begin{equation}\label{4bmv2eqq1}
H_T^{-2\rho(\ll)}(\DDD(\JJ_{\ll}))\simeq H_{T}^{-2\rho(\ll)}(\agl;\DDD(\JJ_{\ll})|_{\agl}) \simeq H_T^0(\agl;\ZZ)\simeq\ZZ.
\end{equation}
By duality, the element of the LHS of \eqref{4bmv2eqq1} corresponding to $1\in\ZZ$ induces a morphism $f_{\ll}:\JJ_{\ll}\ra\DDD_{\agll}[-2\rho(\ll)]$ in the $T$-equivariant derived category.

Let $\iota:\ol{S_{\mu}^-}\cap\agll\hookrightarrow \agll$ denote the inclusion. The morphism $f_{\ll}$ and the counit $\cu:\iota_!\iota^!\ra\id$ give rise to the commutative diagram
\begin{equation}\label{4beqnew}\nonumber
\begin{tikzpicture}
\tikzmath{\x1 = 7; \x2 = 2;}
\node (A) at (0,0) {$H^{2\rho(\mu)}_{T,\ol{S^-_{\mu}}}(\JJ_{\ll})$} ;
\node (B) at (\x1,0) {$H_T^{2\rho(\mu)}(\JJ_{\ll})$} ;
\node (C) at (0,-\x2) {$H^T_{2\rho(\ll-\mu)}(\ol{S^-_{\mu}}\cap\agll)$} ;
\node (D) at (\x1,-\x2) {$H^T_{2\rho(\ll-\mu)}(\agll)$} ;

\node (E) at (-0.32*\x1,0) {$H^{2\rho(\mu)}_{S^-_{\mu}}(\JJ_{\ll})\simeq$} ;

\node (F) at (-0.57*\x1,-\x2) {$H_{2\rho(\ll-\mu)}(\ol{S^-_{\mu}}\cap\agll)\simeq$} ;

\path[->, font=\tiny] (A) edge node[above]{$H_T^{2\rho(\mu)}(\cu(\JJ_{\ll}))$} (B);
\path[->, font=\tiny] (A) edge node[left]{$H_T^{2\rho(\mu)}(\iota^!(f_{\ll}))$} (C);
\path[->, font=\tiny] (B) edge node[right]{$H_T^{2\rho(\mu)}(f_{\ll})$} (D);
\path[->, font=\tiny] (C) edge node[above]{$H_T^{2\rho(\mu-\ll)}(\cu(\DDD_{\agll}))$} (D);
\end{tikzpicture}.
\end{equation}
We claim that $H_T^{2\rho(\mu)}(\iota^!(f_{\ll}))$ sends $v_Z$ to $[Z]$. Let $j_{\ll}:\agl\hookrightarrow \agll$ be the inclusion. There exists a canonical morphism $\JJ_{\ll}\ra (j_{\ll})_*\ZZ[2\rho(\ll)]$ which we shall denote by $h_{\ll}$, and the isomorphism from Lemma \ref{2clemma2} is given by $H_{S^-_{\mu}}^{2\rho(\mu)}(h_{\ll})$. See the proof of \cite[Proposition 11.1]{Satakeluminy}. By our choice of $f_{\ll}$, $j_{\ll}^*(f_{\ll})\circ (j_{\ll}^*(h_{\ll}))^{-1}$ is equal to the identity, after we identify $\DDD_{\agl}[-2\rho(\ll)]$ with $\ZZ[2\rho(\ll)]$ using the canonical orientation on $\agl$. By excision, we get our claim.

The proof is complete if we can show 
\[ H_T^{2\rho(\mu)}(f_{\ll})\circ H_T^{2\rho(\mu)}(\cu(\JJ_{\ll}))(v_Z) = \Phi_{YZ}(\tt(\JJ_{\ll},v_Z,v_{\ll}^*)).\]
It is not difficult to see that for any $v\in H_T^{\bl}(\JJ_{\ll})$,
\[ H_T^{\bl}(f_{\ll})(v)=\sum_i\langle p_{\ll}, h^i\cup v\rangle h_i\]
where $p_{\ll}:H_T^{\bl}(\JJ_{\ll})\ra H_T^{\bl-2\rho(\ll)}(\pt)$ is induced by $f_{\ll}$ via adjunction. On the other hand, 
\[ \Phi_{YZ}(\tt(\JJ_{\ll},v_Z,v_{\ll}^*)) =\sum_i\langle v^*_{\ll}, h^i\cup v\rangle h_i \]
by \eqref{4aeq}. Hence it suffices to show $p_{\ll}^0:=p_{\ll}|_{H^{\bl}(\JJ_{\ll})}=v_{\ll}^*$ (we have $H_T^{\bl}(\JJ_{\ll})\simeq H^{\bl}(\JJ_{\ll})\otimes H_T^{\bl}(\pt)$ by \cite[Lemma 2.2]{YZ}). We have to show (1) $p_{\ll}^0|_{H_{S^-_{\nu}}^{2\rho(\nu)}(\JJ_{\ll})}\equiv 0$ for any $\nu\ne\ll$, and (2) $p_{\ll}^0$ sends $v_e$ to 1. (1) follows from a degree argument. (2) follows from the last paragraph (replace $Z$ by $\{t^{\lambda}\}$).
\end{proof}

\begin{corollary}\label{4bcor}
If $v$ corresponds to the unique MV cycle $\agll$ of type $\ll$ and weight $w_0(\ll)$, then
\[ \Phi_{YZ}(\tt(\JJ_{\ll},v,v_{\ll}^*))=\bk{\agll}.\]
\hfill$\square$
\end{corollary}

\begin{corollary}\label{4bdivisor}
Let $\ll,\mu\in\co$ with $\ll$ dominant. Suppose $\ol{\BB\cdot t^{\mu}}$ is a divisor of $\agll$. Then there exists $v\in S(\ll)$ such that $\Phi_{YZ}(\tt(\JJ_{\ll},v,v_{\ll}^*))=\left[ ~\ol{\BB\cdot t^{\mu}}~\right]$.
\end{corollary}
\begin{proof}
Notice that $\mu=s_{\a_i}(w_0(\ll))$ for some $1\leqslant i\leqslant r$ such that $\a_i(w_0(\ll))\ne 0$. The result follows from Proposition \ref{4bmv2} and Lemma \ref{4badded}.
\end{proof}
\section{The Peterson schemes}\label{5}
Assume $G$ is simple and simply-connected. We are going to define $\YY$ and $\YYs$ in Section \ref{5a} and verify several properties in Section \ref{5ab}. Unless otherwise specified, the coefficient ring $\kkk$ is $\ZZ[\ell_G^{-1}]$ where $\ell_G$ is equal to 1 for type ADE, 2 for type BCF and 3 for type G. Put $S:=\R[q_1,\ldots,q_k]$. The spectrum of $S$ will serve as the base scheme of $\YY$ and $\YY^*$.

In this section and the next, we will complete the following commutative diagram
\begin{equation}\label{1adiag}\nonumber
\begin{tikzpicture}
\tikzmath{\x1 = 2.5; \x2 = 4.3; \x3=4.3; \x4=2.5; \x5=6;}
\node (A) at (0,0) {$\OO(\mathcal{Y}^*)$} ;
\node (B) at (0,-\x1) {$\OO(\Bed)$} ;
\node (C) at (-\x2,0) {$\OO(\mathcal{Y}^*\tGB\Udd)$} ;
\node (D) at (-\x2,-\x1) {$\OO(\Bed)[g_P^{-1}]$} ;
\node (E) at (-\x2-\x3,0) {$\OO(\YYs)$};
\node (F) at (-\x2-\x3,-\x1) {$\OO(\Bed)[g_P^{-1}]/\langle J_P\rangle$};
\node (G) at (0,-\x1-\x4) {$\Hmag$};
\node (H) at (-\x2-\x3,-\x1-\x4) {$\QHPi$};
\node (I) at (-\x2-\x3-\x5,-\x1-\x4) {$\QHP$};
\node (J) at (-\x2-\x3-\x5,0) {$\OO(\YY)$};
\node (K) at (-\x2-\x3-0.5*\x5,-0.5*\x1) {$\OO(\YY)[\ecc]$};

\node (L) at (-\x2-\x3-0.83*\x5,-0.29*\x1) {{\tiny $\etii$}};

\node (M) at (-\x2-\x3-0.3*\x5,-0.18*\x1) {{\small  $\vpi$}};

\node (N) at (-\x2-\x3-0.3*\x5,-0.82*\x1) {{\small  $\vpiv$}};

\path[->, font=\tiny]
(A) edge node[left]{$\etiv$} (B);

\path[->, font=\tiny]
(A) edge node[above]{$\etv$} (C);

\path[->, font=\tiny]
(B) edge node[above]{$\etvi$} (D);

\path[->]
(B) edge node[left]{$\Phi_{YZ}$} (G);

\path[->, font=\tiny]
(C) edge node[right]{$\etvii$} (D);

\path[->, font=\tiny]
(C) edge node[above]{$\etviii$} (E);

\path[->, font=\tiny]
(D) edge node[above]{$\etix$} (F);

\path[->, font=\small]
(E) edge node[right]{$\vpii$} (F);

\path[->, font=\tiny]
(E) edge node[left]{} (K);

\path[->, font=\small]
(F) edge node[right]{$\vpiii$} (H);

\path[->]
(G) edge node[above]{$\Phi_{PLS}$} (H);

\path[->, font=\tiny]
(I) edge node[above]{$\etiii$} (H);

\path[->, font=\tiny]
(J) edge node[above]{$\eti$} (E);

\path[->]
(J) edge node[left]{$\Phi$} (I);

\path[->, font=\tiny]
(J) edge node[below]{} (K);

\path[->, font=\tiny]
(K) edge node[below]{} (F);
\end{tikzpicture}
\end{equation}
\noindent The labellings of all ring homomorphisms below will be compatible with this diagram.
\subsection{Definition}\label{5a}
Let $e^T$ be the element from \eqref{4aelement}. Define
\[ \widetilde{\YYY}:=\left\{(g,h)\in\Gd\times\spec\R\left|~g^{-1}\cdot e^T(h)\in \dd{\bb}_-\oplus\bigoplus_{i=1}^r\dd{\gg}_{\ad_i} \right.\right\}.\]
One checks easily that $\widetilde{\YYY}$ descends to a closed subscheme $\YYY$ of $\GB\times\spec\R$. Define
\[ \YYY^*:=\YYY\tGB \Ud\Bmd/\Bmd.\]
Let $P$ be the parabolic subgroup of $G$ fixed in Section \ref{2a}. Denote by $\Ud_{\Pd_-}\subseteq \Ud_-$ the product, with respect to any order, of the $\ad$-root groups $\Ud_{\ad}\simeq\ga$ (see \cite[Theorem 4.1.4]{BC}) where $\a$ runs over $-(R^+\sm R^+_P)$. By \cite[Proposition 5.1.16]{BC}, it is a subgroup independent of the order. We have 
\[\Ud_{\Pd_-}\simeq \Ud_{\Pd_-}\wpd\Bmd/\Bmd=\Ud_-\wpd\Bmd/\Bmd.\]
 Define
\[ \YY:=\YYY\tGB\Ud_{\Pd_-}\wpd\Bmd/\Bmd\]
and 
\[ \YY^*:=\YY\tGB\Ud\Bmd/\Bmd\simeq\YYY^*\tGB\Ud_{\Pd_-}\wpd\Bmd/\Bmd.\]
Notice that $\YYY^*$, $\YY$ and $\YY^*$ are affine schemes over $\spec\R$.

For any $\a\in R$, define a $\kkk$-linear map
\[ e^*_{\ad}:\gg^{\vee}\ra \gg^{\vee}_{\ad}\xrightarrow{\sim} \kkk\]
where the first arrow is the canonical projection and the second arrow is a $\kkk$-linear isomorphism. We require that $\langle e^*_{\ad_i}, e_{\ad_i}\rangle = 1$ for any $1\leqslant i\leqslant r$ where $e_{\ad_i}$ is the generator of $\gg^{\vee}_{\ad_i}$ fixed when we introduced $e^T$ in Section \ref{4a}.

\begin{definition}\label{5aqdef}
For any $1\leqslant i\leqslant r$, define $\ol{q}_i\in\OO(\YY)$ to be the pull-back, via the natural morphism 
\[\YY\hookrightarrow \Ud_{\Pd_-}\wpd\Bmd/\Bmd\times\spec\R,\]
of the regular function 
\[ (u\wpd\Bmd,h)\mapsto  |\al_i|^{-2}\langle e^*_{-w_0\ad_i},(u\wpd)^{-1}\cdot e^T(h)\rangle,\]
where the representatives $\dot{w}_P$ and $\dot{w}_0$ are defined in Definition \ref{appdef}.
\end{definition}

The map $q_i\mapsto \ol{q}_i$, $1\leqslant i\leqslant k$, makes $\OO(\YY)$ an algebra over $S:=\R[q_1,\ldots,q_k]$, or equivalently, $\YY$ a scheme over $\spec S$. The other $\ol{q}_i$'s are not so useful because of the following
\begin{lemma}\label{5aq=1}
For any $k+1\leqslant i\leqslant r$, we have $\ol{q}_i=1$.
\end{lemma}
\begin{proof}
For such $i$, $-w_P\ad_i$ is equal to a simple coroot $\ad_j$ which has the same length as $\ad_i$. Since $\Ud_{\Pd_-}\subseteq\Ud_-$, we have $u^{-1}\cdot e^T(h)\in e+\dd{\bb}_-$ for any $u\in \Ud_{\Pd_-}$, and hence
\[  \ol{q}_i(u\wpd\Bmd,h)= \pm|\al_i|^{-2}\langle e^*_{\ad_j},u^{-1}\cdot e^T(h)\rangle = \pm|\al_i|^{-2}|\al_j|^{2}=\pm 1.\]
The sign is $+ 1$ by Proposition \ref{appprop1}.
\end{proof}

\subsection{Properties}\label{5ab}
Let 
\[ \eti:\OO(\YY)\ra \OO(\YYs)\]
be the map induced by the immersion $\YYs\hookrightarrow \YY$ and
\[\etii:\OO(\YY)\ra \OO(\YY)[\ecc]\]
the localization map.
\begin{lemma}\label{5alocal}
We have $\etii=\vpi\circ\eti$ for some homomorphism of $\R$-algebras
\[ \vpi:\OO(\YYs)\ra \OO(\YY)[\ecc].\]
\end{lemma}
\begin{proof}
By Lemma \ref{2dlemma3}, $\OO(\YYs)$ is a localization of $\OO(\YY)$ by a regular function. Hence it suffices to show that $\ol{q}_1\cdots\ol{q}_k$ vanishes on $\YY\sm\YY^*$. We may assume $\kkk$ is an algebraically closed field (of characteristic not dividing $\ell_G$). Let $(y,h)$ be a closed point of $\YY\sm\YY^*$. Write $y=u\wpd\Bmd$ with $u\in\Ud_{\Pd_-}$. By the Bruhat decomposition, we can write $u\wpd=b\dot{w}u_1$ for some $b\in\Bd$, $w\in W$ and $u_1\in\Ud_-$. Since $(y,h)\not\in\YY^*$, we have $w\ne e$, and hence there exists $1\leqslant i\leqslant r$ such that $-ww_0\a_i\in -R^+$. It follows that 
\begin{align*}
\ol{q}_i(y,h)&=|\al_i|^{-2}\langle e^*_{-w_0\ad_i},(u\wpd)^{-1}\cdot e^T(h)\rangle\\
&= |\al_i|^{-2}\langle e^*_{-w_0\ad_i},(u\wpd u_1^{-1} )^{-1}\cdot e^T(h)\rangle\qquad \because (y,h)\in\YYY\\
&= \pm |\al_i|^{-2}\langle e^*_{-ww_0\ad_i},b^{-1}\cdot e^T(h)\rangle\\
&=0.
\end{align*}
By Lemma \ref{5aq=1}, we must have $1\leqslant i\leqslant k$, and hence $\ol{q}_1\cdots\ol{q}_k(y,h)=0$ as desired.
\end{proof}

Introduce a $\gm$-action on $\GB$ defined by restricting the canonical $\Gd$-action to the cocharacter $2\rho:\gm\ra \Td$. We also let $\gm$ act on $\spec S$ by taking the usual grading on $\R$ and declaring each $q_i$ to have degree $2\sum_{\a\in R^+\sm R^+_P}\a(\ad_i)$.
\begin{lemma}\label{5aeq}
The $\gm$-action on $\GB\times \spec\R$ preserves and hence induces a $\gm$-action on the schemes $\YYY$, $\YYY^*$, $\YY$ and $\YY^*$. The structure morphism $\YY\ra\spec S$ is $\gm$-equivariant.
\end{lemma}
\begin{proof}
This is straightforward and left to the reader.
\end{proof}

This $\gm$-action is crucial for verifying the following properties of $\YY$.
\begin{lemma}\label{5astratum}
Suppose $\kkk$ is an algebraically closed field of characteristic not dividing $\ell_G$. Then
\[ |\YYY|_{\kkk} =\bigcup_Q |\YYY_Q|_{\kkk}\]
where $|\cdot|_{\kkk}$ denotes the set of $\kkk$-points and $Q$ runs over all standard parabolic subgroups of $G$.
\end{lemma}
\begin{proof}
By flowing every point of $\YYY$ to a $\gm$-fixed point and the Bruhat decomposition, it suffices to show
\[ (\dot{w}\Bmd,0)\in \YYY\implies w=w_Qw_0\text{ for some }Q.\]
By the definition of $\YYY$, $(\dot{w}\Bmd,0)\in \YYY$ implies that for any $1\leqslant i\leqslant r$, $w^{-1}\ad_i$ is either negative or equal to a simple coroot. Denote by $I_Q$ the set of $i\in I:= \{1,\ldots, r\}$ for which the second case holds, and by $Q$ the associated standard parabolic subgroup. Define $w':=w_0w^{-1}w_Q$. We have to show $w'=e$. It suffices to show that $w'\ad_i$ is positive for any $i\in I$.

\bigskip
\underline{\textit{Case 1: $i\in I_Q$.}} We have $w_Q\ad_i=-\ad_j$ for some $j\in I_Q$, and hence $w^{-1}w_Q\ad_i$ is equal to minus a simple coroot, and hence $w'\ad_i$ is positive.

\bigskip
\underline{\textit{Case 2: $i\in I\setminus I_Q$.}} Suppose to the contrary that $w'\ad_i$ is negative. Write $w_Q\ad_i=\sum_{j\in I}c_j\ad_j$. Then $c_j\geqslant 0$ for each $j$ and is non-zero for some $j_0\in I\setminus I_Q$. Since $w'\ad_i$ is negative, the sum $\sum_{j\in I}c_j w_0w^{-1}\ad_j=w'\ad_i$ is a negative coroot. By the definition of $I_Q$, $w_0w^{-1}\ad_j$ is positive (resp. minus a simple coroot) for $j\in I\setminus I_Q$ (resp. $j\in I_Q$). It follows that $w_0w^{-1}\ad_{j_0}$ lies in the $\ZZ$-span of $w_0w^{-1}\ad_j$ with $j\in I_Q$, or equivalently, $\ad_{j_0}$ lies in the $\ZZ$-span of $\ad_j$ for these $j$, a contradiction.
\end{proof}

\begin{lemma}\label{5avf}
Suppose $\kkk$ is an algebraically closed field of characteristic not dividing $\ell_G$. The fiber $(\YY)_x$ of $\YY$ over $x:=(0,0)\in\spec S$ is isomorphic to the zero locus of a vector field on $\Gd/\Pbd$, where
$\Pbd$ is the parabolic subgroup of $\Gd$ containing $\Bmd$ such that
\[ \pbd:=\lie(\Pbd)=\dd{\bb}_-\oplus\bigoplus_{\a\in -w_0R^+_P}\dd{\gg}_{\ad}.\]
Set-theoretically, it is equal to $\{(\wpd\Bmd,0)\}$.
\end{lemma}
\begin{proof}
We have $\TTT_{\Gd/\Pbd}\simeq \Gd\times_{\Pbd}(\dd{\gg}/\pbd)$ so the morphism $g\mapsto g^{-1}\cdot e^T(0)\bmod\pbd$ defines a section $s_{e^T(0)}\in H^0(\Gd/\Pbd;\TTT_{\Gd/\Pbd})$. Observe that $\Ud_{\Pd_-}\simeq \Ud_{\Pd_-}\wpd\Pbd/\Pbd$ is an open subscheme of $\Gd/\Pbd$. We first show that $(\YY)_x$ is isomorphic to the zero locus of the restriction of $s_{e^T(0)}$ to this subscheme. The latter scheme is by definition the closed subscheme of $\Ud_{\Pd_-}$ cut by the polynomials $\langle e^*_{\ad}, (u\wpd)^{-1}\cdot e^T(0)\rangle$ for $\a\in -w_0(R^+\sm R^+_P)$ which are part of the defining polynomials of $(\YY)_x$. The other defining polynomials are those corresponding to $\a\in -w_0R^+_P\sm\{-w_0\a_{k+1},\ldots,-w_0\a_r\}$. It is not hard to see that they are all equal to zero and so redundant. This justifies our claim.

To complete the proof, it suffices to show that the set of closed points of the zero locus of $s_{e^T(0)}$ is equal to $\{\wpd\Pbd\}$. Observe that the former set is $\gm$-invariant. Hence it suffices to show that $\wpd\Pbd$ is the only $\gm$-fixed point of $\Gd/\Pbd$ at which $s_{e^T(0)}$ vanishes. Clearly, $s_{e^T(0)}$ vanishes at $\wpd\Pbd$. Pick a point $b$ in the dominant chamber such that $\a_i(b)=0$ if $\a_i\in -w_0R^+_P$ and $\a_i(b)>0$ otherwise. Let $w\in W$. If $w\not\in \wp W_{\Pbd}=w_0W_{\Pbd}$, then $w(b)\ne w_0(b)$, and hence $w(b)$ is away from the anti-dominant chamber. It follows that there exists $1\leqslant i\leqslant r$ such that $\a_i(w(b))>0$. This implies $w^{-1}\a_i\in -w_0(R^+\sm R^+_P)$, and hence $s_{e^T(0)}(\dot{w}\Pbd)\ne 0$. 
\end{proof}

\begin{lemma}\label{5aflat}
$\YY$ is flat over $\spec S$.
\end{lemma}
\begin{proof}
It suffices to verify the flatness at every closed point of $\YY$. Let $y\in \YY$ be a closed point lying over a point $x\in\spec S$. By a generalization of the Nullstellensatz, $x$ is a closed point. By \cite[Proposition 6.1.5]{EGA2}, $\OO_{\YY,y}$ is a flat $\OO_{\spec S,x}$-module if the following conditions are satisfied:
\begin{enumerate}
\item $\OO_{\spec S,x}$ is regular;

\item $\OO_{\YY,y}$ is Cohen-Macaulay; and

\item $\dim\OO_{\YY,y}=\dim_y(\YY)_x + \dim\OO_{\spec S,x}$ (in general LHS $\leqslant$ RHS).
\end{enumerate}
(1) is obvious. Observe that $\YY$ is a closed subscheme of $\Ud_{\Pd_-}\times\spec S\simeq \mathbb{A}^{N}_S$ ($N:=\dim\Ud_{\Pd_-}$) defined by $N$ polynomials, namely $\langle e^*_{\ad},(u\wpd)^{-1}\cdot e^T(h)\rangle$ for $\a\in -w_0(R^+\sm R^+_P)\sm\{\a_1,\ldots,\a_r\}$ and $|\ad_i|^{-2}\langle e^*_{-w_0\ad_i},(u\wpd)^{-1}\cdot e^T(h)\rangle -q_i$ for $1\leqslant i\leqslant k$. It follows that 
\[\dim\OO_{\YY,y}\geqslant \dim S=\dim\OO_{\spec S,x}.\] To complete the proof, it suffices to show that $\YY$ has zero-dimensional fibers, for we will have $\dim\OO_{\YY,y}=\dim\OO_{\spec S,x}$, yielding (2) and (3) simultaneously. 

We may now assume $\kkk$ is an algebraically closed field of characteristic not dividing $\ell_G$. Consider $y_0:=(\wpd \Bmd,0)\in \YY$ which lies over $x_0:=(0,0)\in\spec S$. By Lemma \ref{5avf}, $(\YY)_{x_0}$ is set-theoretically equal to $\{y_0\}$, and hence, by \cite[Th\'eor\`eme 13.1.3]{EGA3}, the fiber dimension at every point near $y_0$ is equal to zero. The result then follows from Lemma \ref{5aeq} and the observation that the $\gm$-action takes every closed point of $\YY$ to $y_0$ at the limit.
\end{proof}

\begin{lemma}\label{5aproper}
Suppose $\kkk=\CC$. $\YY$ is proper over $\spec S$.
\end{lemma}
\begin{proof}
Define $\YY'\subset \YY\times \spec\kkk[q_1,\ldots,q_k] $ to be the graph of the morphism
\begin{equation}\nonumber
\begin{array}{ccc}
\YY & \ra & \spec\kkk[q_1,\ldots,q_k] \\ [.5em]
(u\wpd\Bmd, h) &\mapsto & \left(\ol{q}_1(u\wpd\Bmd,h),\ldots,\ol{q}_k(u\wpd\Bmd,h) \right)
\end{array}
\end{equation}
where $\ol{q}_1,\ldots,\ol{q}_k$ are defined in Definition \ref{5aqdef}. To prove that $\YY$ is proper over $\spec S$, it suffices to show that the natural morphism $\YY'\ra \GB\times\spec S$ is a closed immersion. Notice that $\YY'$ is a closed subscheme of $\wpd\Ud\Bmd/\Bmd\times\spec S$ which is a divisor complement of $\GB\times\spec S$ by Lemma \ref{2dlemma3}. Hence, by Lemma \ref{5aproperfact1}, it suffices to show that $\YY'$ is a closed subset of $\GB\times\spec S$.

Suppose to the contrary that $\YY'$ is not closed in $\GB\times\spec S$. Let $\ol{\YYY}'_P$ be the (set-theoretic) closure of $\YY'$. Then $\ol{\YYY}'_P\setminus \YY'\ne \emptyset$. Since $\YY'$ is locally closed, $\ol{\YYY}'_P\setminus \YY'$ is closed. By Lemma \ref{5aeq}, the given $\gm$-action on $\GB\times\spec S$ preserves $\YY'$ and so $\ol{\YYY}'_P\setminus \YY'$. Consider a $\gm$-fixed point of $\ol{\YYY}'_P\sm\YY'$ which, by Lemma \ref{5astratum}, must be of the form $(\dot{w}_Q\dot{w}_0\Bmd,0,0)$ for some standard parabolic subgroup $Q$ of $G$. Observe that $\ol{\YYY}'_P\sm \YY'\subseteq \left(\ol{\Ud_{\Pd_-}\wpd\Bmd/\Bmd}\sm \Ud_{\Pd_-}\wpd\Bmd/\Bmd\right)\times \spec S$, and hence $P\subsetneq Q$. WLOG, we assume that the simple roots of $Q$ are $\a_{\ell+1},\ldots,\a_r$ for some $\ell< k$. 

Let $(u\wpd\Bmd,h)\in \YY\tGB\dot{w}_Q\dot{w}_0\Ud\Bmd/\Bmd$ with $u\in\Ud_{\Pd_-}$. There are unique $u_1\in\Ud$, $u_2\in\Ud_-$ and $t\in \Td$ such that 
\begin{equation}\label{5apropereq1}
u\wpd=\dot{w}_Q\dot{w}_0u_1tu_2.
\end{equation}
For any $1\leqslant i\leqslant r$, we have
\begin{align}\label{5aeeqa}
& \ol{q}_i(u\wpd\Bmd,h)\nonumber\\
=~&|\al_i|^{-2}\langle e^*_{-w_0\ad_i},(u\wpd)^{-1}\cdot e^T(h)\rangle \nonumber\\ 
=~&|\al_i|^{-2}\langle e^*_{-w_0\ad_i},(u\wpd u_2^{-1})^{-1}\cdot e^T(h)\rangle\qquad\qquad\because (u\wpd\Bmd,h)\in \YY \nonumber\\ 
=~&|\al_i|^{-2}\ad_i(w_0(t))\langle e^*_{-w_0\ad_i},(\dot{w}_Q\dot{w}_0u_1)^{-1}\cdot e^T(h)\rangle \nonumber\\
=~&\ad_i(w_0(t))f_i(u_1,h) 
\end{align}
where 
\[f_i:(u_1,h)\mapsto |\al_i|^{-2}\langle e^*_{-w_0\ad_i},(\dot{w}_Q\dot{w}_0u_1)^{-1}\cdot e^T(h)\rangle\] is a regular function on $\Ud\times\spec\R$. Notice that $f_i(e,0)\ne 0$ if $\ell+1\leqslant i\leqslant r$.

Let $\ll\in\Q$ be dominant. Recall the elements $v_{\wp}\in S(\ll)_{\wp(\ll)}$ and $v_{\ll}^*\in S(\ll)^*$ defined in Definition \ref{2ddef}. Define $h_{\ll}\in\OO(\Ud)$ by
\[ h_{\ll}(u_1):=\langle v_{\ll}^*, (\dot{w}_Q\dot{w}_0u_1)^{-1}\cdot v_{\wp}\rangle.\] 
Suppose \eqref{5apropereq1} continues to hold. Put $u_3:=(\wpd)^{-1}u(\wpd)$. We have $u_3\in\Ud$ and 
\begin{equation}\label{5aeeqb}
h_{\ll}(u_1)=\langle v_{\ll}^*, tu_2u_3^{-1}(\wpd)^{-1}\cdot v_{\wp}\rangle= \langle v_{\ll}^*, tu_2u_3^{-1}\cdot v_e\rangle = \ll(t)
\end{equation}
where the second equality follows from Proposition \ref{appprop2}. Let $\{\od_1,\ldots,\od_r\}$ be the dual basis of $\{\a_1,\ldots,\a_r\}$. By Lemma \ref{5aproperfact2}, we can find $\ll_i\in\Q\cap\RR_{>0}\cdot\od_i$ for each $1\leqslant i\leqslant \ell+1$ such that 
\begin{equation}\label{5aeeqc}
\ll_{\ell+1}=\sum_{i=1}^{\ell}b_i\ll_i+\sum_{j=\ell+1}^rc_j\ad_j 
\end{equation}
for some non-negative integers $b_1,\ldots, b_{\ell}$ and $c_{\ell+1},\ldots,c_r$. By \eqref{5aeeqa}, \eqref{5aeeqb} and \eqref{5aeeqc}, we have
\begin{align*}
h_{-w_0(\ll_{\ell+1})}(u_1) =~& (-w_0(\ll_{\ell+1}))(t)\\
 =~&  \prod_{i=1}^{\ell} (-w_0(\ll_i))(t)^{b_i}\times \prod_{j=\ell+1}^r\ad_j(w_0(t^{-1}))^{c_j}\\
=~& \prod_{i=1}^{\ell} h_{-w_0(\ll_i)}(u_1)^{b_i}\times \prod_{j=\ell+1}^r \frac{f_j(u_1,h)^{c_j}}{\ol{q}_j(u\wpd\Bmd,h)^{c_j}} .
\end{align*}
Observe that for any $1\leqslant i\leqslant \ell+1$,
\begin{align*}
h_{-w_0(\ll_i)}(e)\neq 0~  \Longleftrightarrow~&~ w_Qw_P(\ll_i)=\ll_i\\
\Longleftrightarrow~&~ s_{\a_i}\text{ does not appear in a reduced}\\
&~\text{word decomposition of }w_Qw_P
\\
\Longleftrightarrow~&~ 1\leqslant i\leqslant\ell.
\end{align*}

To summarize, we have proved that $h_{-w_0(\ll_{\ell+1})}\times \prod_{j=\ell+1}^r q_j^{c_j}$ and $\prod_{i=1}^{\ell} h_{-w_0(\ll_i)}^{b_i}\times \prod_{j=\ell+1}^r f_j^{c_j}$, considered as regular functions on the open subscheme $\dot{w}_Q\dot{w}_0\Ud\Bmd/\Bmd\times\spec S$ of $\GB\times\spec S$, agree at every closed point of $\YY'\tGB \dot{w}_Q\dot{w}_0\Ud\Bmd/\Bmd$ but do not agree at the limit point $(\dot{w}_Q\dot{w}_0\Bmd,0,0)$, a contradiction. (We have used the well-known fact that any finitely generated algebra over a field is a Jacobson ring.) The proof of Lemma \ref{5aproper} is complete.
\end{proof}

\begin{proposition}\label{5arank}
The dimension of $\OO(\YY)\otimes_S\fof S$ is equal to $|W/W_P|$.
\end{proposition}
\begin{proof}
We may assume $\kkk=\CC$. Since $\YY$ is quasi-finite (proof of Lemma \ref{5aflat}) and proper (Lemma \ref{5aproper}) over $\spec S$, it follows that, by \cite[Th\'eor\`eme 8.11.1]{EGA3}, $\OO(\YY)$ is a finitely generated $S$-module and hence a projective $S$-module by the flatness (Lemma \ref{5aflat}). (It is well-known that any finitely presented flat module is projective.) The dimension in question is thus equal to the dimension of the coordinate ring of the fiber of $\YY$ over $x:=(0,0)\in\spec S$. By Lemma \ref{5avf}, $(\YY)_x$ is isomorphic to the zero locus of a vector field on $\Gd/\Pbd$, and hence
\[ \dim_{\CC}\OO((\YY)_x)=\chi( \Gd/\Pbd) = |W/W_{\Pbd}| = |W/W_P|. \]
\end{proof}
\section{Proof of main results}\label{6}
\subsection{Construction of homomorphisms}\label{5b}
The goal of this subsection is to define three homomorphisms of $S$-algebras
\[ \Phi:\OO(\YY)\ra \QHP,\]

\[ \Phi^*:\OO(\YY^*)\ra \QHPi\]
and 
\[\Phi_{loc}:\OO(\Bed\times_{\Gd}\Ud_{-}(\wpd)^{-1}\Bmd)\ra \QHPi .\]

Let us start with a morphism of schemes 
\begin{equation}\label{5beq1}
\Bed\ra \GB\times\spec\R
\end{equation}
defined by $(b,h)\mapsto (b^{-1}\Bmd,h)$. Clearly, it factors through the immersion $\YYY^*\hookrightarrow\GB\times\spec\R$ so we obtain a homomorphism of $\R$-algebras
\[ \etiv:\OO(\YYY^*)\ra \OO(\Bed).\] 

For any dominant $\ll\in\Q$, define $g_{P,\ll}\in\OO(\Bed)$ by
\[g_{P,\ll}(b,h):=\tt(\JJ_{\ll},v_{\wp},v_{\ll}^*)(b,h)=\langle v_{\ll}^*, b\cdot v_{\wp}\rangle.\]
(See Definition \ref{2ddef} for the definitions of $v_{\wp}$ and $v_{\ll}^*$.) Recall in the paragraph before Lemma \ref{2dlemma3} we have associated to any $v\in S(\ll)$ a section $s_{\ll,v}\in H^0(\GB;\LL(\ll))$. Observe that the pull-back of $s_{\ll,v_{\wp}}$ via \eqref{5beq1} is equal to $g_{P,\ll}$. 

Now consider the element $\ll_0\in\ZZ_{>0}\cdot 2\rhod$ from Proposition \ref{2dcor}. Define
\[ g_P:= g_{P,\ll_0}.\]
By Lemma \ref{2dlemma3}, the open subscheme $\Udd=\wpd\Ud\Bmd/\Bmd$ of $\GB$ is the complement of the divisor $\{s_{\ll_0,v_{\wp}}=0\}$. Thus we obtain a homomorphism of $\R$-algebras
\[ \etvii: \OO(\YYY^*\tGB\Udd)\ra \OO(\Bed\tGB\Udd)\simeq\OO(\Bed)[g_P^{-1}].\]

\begin{definition}\label{5bideal}
Define $J_P$ to be the ideal of $\OO(\Bed)$ generated by regular functions 
\[ (b,h)\mapsto \tt(\Jl,v,v_{\ll_0}^*)(b,h) = \langle v_{\ll_0}^*,b\cdot v\rangle \]
with $v\in S(\ll_0)_{<\wp(\ll_0)}:= \bigoplus_{\mu<\wp(\ll_0)} S(\ll_0)_{\mu}$. (Here $\mu< \wp(\ll_0)$ means $\mu\leqslant \wp(\ll_0)$ and $\mu\ne \wp(\ll_0)$.)
\end{definition}

\begin{lemma}\label{5bexist1}
We have
\[ \OO(\Bed\times_{\Gd}\Ud_{-}(\wpd)^{-1}\Bmd)\simeq \OO(\Bed\tGB\Ud_{\Pd_-}\wpd\Bmd/\Bmd) \simeq \OO(\Bed)[g_P^{-1}]/\langle J_P\rangle,\]
and hence $\etvii$ induces a homomorphism of $\R$-algebras
\[ \vpii:\OO(\YY^*)\ra \OO(\Bed)[g_P^{-1}]/\langle J_P\rangle.\]
\end{lemma}
\begin{proof}
The first isomorphism is clear. By definition, $\Bed\tGB\Ud_{\Pd_-}\wpd\Bmd/\Bmd$ is a closed subscheme of $\Bed\tGB\Udd$ whose defining ideal is generated by the pull-back of the regular functions on $\Udd$ belonging to the defining ideal of $\Ud_{\Pd_-}\wpd\Bmd/\Bmd$. There exists a $\Td$-equivariant isomorphism
\[\Udd \simeq\prod_{\a\in\wp R^+}\Ud_{\ad}.\]
See the proof of \cite[Theorem 5.1.13]{BC}. Under this isomorphism, $\Ud_{\Pd_-}\wpd\Bmd/\Bmd$ is identified with $\prod_{\a\in\wp R^+\cap (-R^+)}\Ud_{\ad}$. Observe that
\[ \wp R^+\cap (-R^+) = -(R^+\sm R^+_P)\quad\text{ and } \quad\wp R^+\cap R^+= R^+_P.\]
It follows that the defining ideal of $\Ud_{\Pd_-}\wpd\Bmd/\Bmd$ is generated by the homogeneous elements of $\OO(\Udd)$ with respect to the $\Td$-action whose weights belong to $\left(\sum_{\a\in R^+_P}\ZZ_{\leqslant 0}\cdot \ad\right)\sm\{0\}$.

Now, by Proposition \ref{2dcor}, the ring $\OO(\Udd)$ is generated by regular functions of the form $s_{\ll_0,v}/s_{\ll_0,v_{\wp}}$ with $v\in S(\ll_0)$. If $v\in S(\ll_0)_{\mu}$, then the $\Td$-weight of $s_{\ll_0,v}/s_{\ll_0,v_{\wp}}$ is equal to $\mu-\wp(\ll_0)$. It follows that, by Lemma \ref{5bexist1fact1}, the defining ideal of $\Ud_{\Pd_-}\wpd\Bmd/\Bmd$ is generated by $s_{\ll_0,v}/s_{\ll_0,v_{\wp}}$ with $v\in S(\ll_0)_{< \wp(\ll_0)}$. The proof is complete by noticing that the pull-back of $s_{\ll_0,v}/s_{\ll_0,v_{\wp}}$ to $\OO(\Bed)[g_P^{-1}]$ is equal to $g_P^{-1}\cdot\tt(\JJ_{\ll_0},v,v_{\ll_0}^*)$.  
\end{proof}

Recall the ring maps $\Phi_{PLS}$ and $\Phi_{YZ}$ from Theorem \ref{3aPLS} and Theorem \ref{4aYZ} respectively. 
\begin{lemma}\label{5bgp}
For any dominant $\ll\in\Q$, we have $\Phi_{PLS}\circ\Phi_{YZ}(g_{P,\ll})=q^{[w_0(\ll)]}$.
\end{lemma}
\begin{proof}
Let $Z$ be the unique MV cycle of type $\ll$ and weight $\wp(\ll)$. By Proposition \ref{4bmv2}, 
\[ \Phi_{YZ}(g_{P,\ll})=[Z]\in H^T_{2\rho(\ll-\wp(\ll))}(\agll).\]
Consider the subvarieties $Z$ and $\agp:=\ol{L(\OO)\cdot t^{\wp(\ll)}}$ of $\agll$. Inside an open neighbourhood of $t^{\wp(\ll)}$, they are non-singular and intersect each other transversely at this point. Let $t^{\mu}\in Z$. Then $\mu\geqslant \wp(\ll)$ by Lemma \ref{4bclosure}, and hence $t^{\mu}\not\in \agp$ unless $\mu=\wp(\ll)$. This gives $Z\cap \agp=\{t^{\wp(\ll)}\}$. The result thus follows from Proposition \ref{3bone}.
\end{proof}

\begin{lemma}\label{5bjp}
We have $\Phi_{PLS}\circ\Phi_{YZ}(J_P)=\{0\}.$
\end{lemma}
\begin{proof}
Recall the ideal $J_P$ is generated by $\tt(\JJ_{\ll_0},v,v_{\ll_0}^*)$ with $v\in S(\ll_0)_{\mu}$ where $\mu<\wp(\ll_0)$. By Proposition \ref{4bmv2}, $\Phi_{YZ}(\tt(\JJ_{\ll_0},v,v_{\ll_0}^*))$ lies in $H^T_{2\rho(\ll_0-\mu)}(\ag^{\leqslant\ll_0})$. By the assumption on these $\mu$, we have 
\[2\rho(\ll_0-\mu)>2\rho(\ll_0-\wp(\ll_0)).\]
The result thus follows from Proposition \ref{3bzero}.
\end{proof}

\begin{lemma}\label{5bexist2}
The composition $\Phi_{PLS}\circ\Phi_{YZ}$ induces a homomorphism of $\R$-algebras
\[ \vpiii: \OO(\Bed)[g_P^{-1}]/\langle J_P\rangle \ra\QHPi.\]
\end{lemma}
\begin{proof}
This follows immediately from Lemma \ref{5bgp} and Lemma \ref{5bjp}.
\end{proof}

\begin{lemma}\label{5bland}
The composition $\vpiii\circ\vpii\circ\eti$
\[  \OO(\YY)\xrightarrow{\eti} \OO(\YY^*) \xrightarrow{\vpii} \OO(\Bed)[g_P^{-1}]/\langle J_P\rangle \xrightarrow{\vpiii} \QHPi\]
lands in $\QHP$ where $\eti$ is induced by the immersion $\YYs\hookrightarrow \YY$, and $\vpii$ and $\vpiii$ come from Lemma \ref{5bexist1} and Lemma \ref{5bexist2} respectively.
\end{lemma}
\begin{proof}
Notice that $\YY$ is a closed subscheme of $ \Ud(\wp)\times\spec\R$. By Proposition \ref{2dcor}, the ring $\OO(\Udd)$ is generated by regular functions $s_{\ll_0,v}/s_{\ll_0,v_{\wp}}$ with $v\in S(\ll_0)$. Hence it suffices to show that all these regular functions are mapped into $\QHP$. Their images in $\OO(\Bed)[g_P^{-1}]$ are equal to $g_P^{-1}\cdot\tt(\JJ_{\ll_0},v,v_{\ll_0}^*)$. By Lemma \ref{5bgp}, $g_P^{-1}$ goes to $q^{-[w_0(\ll_0)]}$. Since $\Phi_{YZ}(\tt(\JJ_{\ll_0},v,v_{\ll_0}^*))\in \HT(\ag^{\leqslant \ll_0})$ by Proposition \ref{4bmv2}, we have, by Lemma \ref{3bnoempty} or the explicit description of $\Phi_{PLS}$ (Theorem \ref{3aPLS}),
\[  \vpiii(g_P^{-1}\cdot\tt(\JJ_{\ll_0},v,v_{\ll_0}^*))\in\sum_{\eta\in [w_0(\ll_0)]+\ec } H_T^{\bl}(G/P)\otimes q^{\eta-[w_0(\ll_0)]} =\QHP.\]
The proof is complete. 
\end{proof}

Let $\etiii: \QHP\ra \QHP[\ecc]$ be the localization map.
\begin{definition}\label{5bphi} $~$
\begin{enumerate}
\item Define 
\[ \Phi:\OO(\YY)\ra \QHP \]
to be the unique map satisfying $\etiii\circ\Phi =\vpiii\circ\vpii\circ\eti$. (See Lemma \ref{5bland}.)
\item Define 
\[ \Phi^*:\OO(\YY^*)\ra \QHPi \]
to be the composition $\vpiii\circ\vpii$.
\item Define
\[\Phi_{loc}:\OO(\Bed\times_{\Gd}\Ud_{-}(\wpd)^{-1}\Bmd)\ra \QHPi\]
to be $\vpiii$. (We apply Lemma \ref{5bexist1} to identify their domains.)
\end{enumerate}
\end{definition}

\begin{remark}
The idea of composing $\Phi_{YZ}$ and $\Phi_{PLS}$ to get the above maps is not entirely new. See \cite[Remark 6.5]{LR} in which the arguments imply that for $P=B$ the composition induces the isomorphism $\Phi_{loc}$. Our contribution is to prove that for arbitrary $P$ the composition induces all the maps from Definition \ref{5bphi} and that they are bijective. The first part seems already non-trivial as it depends on Lemma \ref{5bgp} which computes $\Phi_{PLS}$ where the input is the fundamental class of a specific MV cycle which is not equal to any affine Schubert class unless $P=B$.
\end{remark}

\begin{lemma}\label{5bhomo}
$\Phi$ and $\Phi^*$ are homomorphisms of $S$-algebras.
\end{lemma}
\begin{proof}
It is clear that they are homomorphisms of $\R$-algebras. To complete the proof , it suffices to show that $\Phi(q_i\cdot 1)=q^{[\ad_i]}$ for any $1\leqslant i\leqslant k$. Recall $q_i$ acts on $\OO(\YY)$ by the regular function 
\[ \ol{q}_i: (u\wpd\Bmd,h)\mapsto |\al_i|^{-2}\langle e^*_{-w_0\ad_i},(u\wpd)^{-1}\cdot e^T(h)\rangle.\]
Denote by $b$ and $u$ the composite morphisms
\[ \spec\OO(\Bed)[g_P^{-1}]/\langle J_P\rangle \hookrightarrow \Bd\times\spec\R \ra \Bd\]
and 
\[ \spec\OO(\Bed)[g_P^{-1}]/\langle J_P\rangle\xrightarrow{\spec(\vpii\circ\eti)} \YY\ra \Ud_{\Pd_-}\wpd\Bmd/\Bmd\simeq \Ud_{\Pd_-} \]
respectively, where the unlabelled arrows are canonical morphisms. There exist morphisms $u_1$ and $t$ from $\spec\OO(\Bed)[g_P^{-1}]/\langle J_P\rangle$ to $\Ud_-$ and $\Td$ respectively such that $u\wpd = b^{-1}u_1 t$. (The twist $b^{-1}$ agrees with the morphism \eqref{5beq1}.) Then
\[ \vpii\circ\eti(\ol{q}_i)= |\ad_i|^{-2} \langle e^*_{-w_0\ad_i}, t^{-1}u_1^{-1}b\cdot e^T(h)\rangle =\ad_i(w_0(t)).\]

On the other hand, for any dominant $\ll\in\Q$, we have 
\[g_{P,\ll}=\tt(\JJ_{\ll},v_{\wp},v_{\ll}^*)=\langle v_{\ll}^*, u_1t(\wpd)^{-1}u^{-1}\cdot v_{\wp}\rangle=\langle v_{\ll}^*, u_1t(\wpd)^{-1}u^{-1}(\wpd)\cdot v_{e}\rangle\]
where the last equality follows from Proposition \ref{appprop2}. Observe that $(\wpd)^{-1} \Ud_{\Pd_-}(\wpd)\subseteq \Ud$. It follows that $(\wpd)^{-1}u^{-1}(\wpd)$ lands in $\Ud$, and hence $g_{P,\ll}=\ll(t)$. 

Now, pick dominant $\ll_1, \ll_2\in\Q$ such that $\ll_1=\ll_2+w_0\ad_i$. It follows that
\[ \vpii\circ\eti(\ol{q}_i)=\ad_i(w_0(t))=\ll_1(t)/\ll_2(t)=g_{P,\ll_1}/g_{P,\ll_2},\]
and hence, by Lemma \ref{5bgp},
\[ \Phi(q_i\cdot 1)=\Phi_{PLS}\circ\Phi_{YZ}(g_{P,\ll_1}/g_{P,\ll_2}) = q^{[w_0(\ll_1-\ll_2)]}=q^{[\ad_i]}.\]
\end{proof}

\begin{lemma}\label{5bhomoloc}
$\Phi_{loc}$ is a homomorphisms of $S$-algebras.
\end{lemma}
\begin{proof}
Let us preserve the notation from the proof of Lemma \ref{5bhomo}. In particular, we have $u\wpd = b^{-1}u_1 t$ which implies
\[ b= u_1(\wpd)^{-1}((\wpd)t(\wpd)^{-1})u^{-1}.\]
It follows that $\ol{q}_{loc,i}=\ad_i(\wp(t))$ for any $i$. (By Lemma \ref{5bexist1}, we can view $\ol{q}_{loc,i}$ as an element of $\OO(\Bed)[g_P^{-1}]/\langle J_P\rangle$.) We have proved $\vpii\circ\eti(\ol{q}_i)=\ad_i(w_0(t))$ for any $1\leqslant i\leqslant r$. It follows that, by Lemma \ref{5aq=1}, $\ad_i(w_0(t))=1$ for any $k+1\leqslant i\leqslant r$, and hence $\wp(t)=w_0(t)$. This gives 
\[ \ol{q}_{loc,i}= \ad_i(\wp(t))=\ad_i(w_0(t))=\vpii\circ\eti(\ol{q}_i)\]
which goes to $q^{[\ad_i]}$ via $\vpiii$.
\end{proof}


\subsection{Proof of Theorem \ref{1athm}}\label{6a}
Recall we have defined $\Phi$ and $\Phi^*$ in Definition \ref{5bphi}. By Lemma \ref{5bhomo}, they are homomorphisms of $S$-algebras. Theorem \ref{1athm} follows immediately from
\begin{proposition}\label{6athma}
$\Phi$ is bijective.
\end{proposition}

\begin{proposition}\label{6athmb}
$\Phi^*$ is injective and is bijective at the non-equivariant limit.
\end{proposition}

We prove some lemmas first.
\begin{lemma}\label{6asurj1}
The localization map
\[\Phi[\ecc]:\OO(\YY)[\ecc]\ra \QHPi\]
is surjective.
\end{lemma}
\begin{proof}
By the proof of Lemma \ref{5bhomo}, $\vpii\circ\eti(\ol{q}_i)$ is a unit for any $1\leqslant i\leqslant k$. Hence there exists
\[ \vpiv:\OO(\YY)[\ecc]\ra \OO(\Bed)[g_P^{-1}]/\langle J_P\rangle\]
such that $\vpiv\circ\etii=\vpii\circ\eti$. It follows that $\Phi[\ecc]=\vpiii\circ\vpiv$. By the explicit description of $\Phi_{PLS}$ (Theorem \ref{3aPLS}), the image of $\Phi_{PLS}$ generates $\QHPi$ as an $S[\ecc]$-module. Since $\Phi_{YZ}$ is surjective (Theorem \ref{4aYZ}), $\Phi[\ecc]$ is surjective if we can show that $\vpiv$ is.

By definition, $\Bed$ is a closed subscheme of $\Bd\times\spec\R\simeq \Td\times\Ud\times\spec\R$. Hence it suffices to show that $\im\vpiv$ contains (1) $g_P^{-1}$; (2) all regular functions from $\OO(\Ud)$; and (3) all regular functions from $\OO(\Td)$. 

\bigskip
\begin{enumerate}
\item In fact we will prove (1') $g_{P,\ll}^{\pm 1}\in \im\vpiv$ for any dominant $\ll\in\Q$. Let $\{\w_1,\ldots,\w_r\}$ be the dual basis of $\{\ad_1,\ldots,\ad_r\}$. Then $-\ll = \sum_{i=1}^r( -\w_i(w_0(\ll)))w_0(\ad_i)$. Since the dominant chamber is contained in the closed cone spanned by the simple roots (see \cite[\S 3 no.5 Lemme 6]{Bourbaki}), we have $-\w_i(w_0(\ll))\geqslant 0$ for each $i$. By the proof of Lemma \ref{5bhomo}, $\vpii\circ\eti$ sends 
\[ \ol{q}_1^{-\w_1(w_0(\ll))}\cdots  \ol{q}_r^{-\w_r(w_0(\ll))} = \ol{q}_1^{-\w_1(w_0(\ll))}\cdots  \ol{q}_k^{-\w_k(w_0(\ll))}\in\OO(\YY)\]
to $g_{P,\ll}^{-1}$. It follows that $g_{P,\ll}^{-1}\in \im\vpii$. By Lemma \ref{5alocal}. which implies $\vpii=\vpiv\circ\vpi$, we have $g_{P,\ll}^{-1}\in\im \vpiv$. But since we can invert each $\ol{q}_i$ in $\OO(\YY)[\ecc]$, we also have $g_{P,\ll}\in\im \vpiv$.

\item[]

\item Since the composition $\Bed\xrightarrow{\pr}\Ud\xrightarrow{u~\mapsto ~u^{-1}} \Ud\simeq \Ud\Bmd/\Bmd$ factors through $\YYY^*$, the image of the composite map
\[ \OO(\Ud) \xrightarrow{\OO(\pr)} \OO(\Bed) \ra \OO(\Bed)[g_P^{-1}]/\langle J_P\rangle\]
is contained in $\im\vpii$ and hence in $\im\vpiv$ by Lemma \ref{5alocal}.

\item[] 

\item Since $\OO(\Td)$ is generated by the characters of $\Td$ and the character associated to any dominant $\ll\in\Q$ corresponds to $\tt(\JJ_{\ll},v_e,v_{\ll}^*)$, it suffices to show that $\im\vpiv$ contains these regular functions and their inverses. By Lemma \ref{2dlemma3}, both $s_{\ll,v_e}/s_{\ll,v_{\wp}}$ and $s_{\ll,v_{\wp}}/s_{\ll,v_e}$ are regular functions on $\Udd\cap\Ud\Bmd/\Bmd$. They give rise to elements of $\OO(\YYs)$ which are sent by $\vpii$ to $g_{P,\ll}^{-1}\cdot\tt(\JJ_{\ll},v_e,v_{\ll}^*)$ and $g_{P,\ll}\cdot\tt(\JJ_{\ll},v_e,v_{\ll}^*)^{-1}$ respectively. It follows that, by Lemma \ref{5alocal}, we have $g_{P,\ll}^{\pm 1}\cdot\tt(\JJ_{\ll},v_e,v_{\ll}^*)^{\mp 1}\in\im\vpiv$, and hence, by (1'), $\tt(\JJ_{\ll},v_e,v_{\ll}^*)^{\pm 1} \in\im\vpiv$ as desired. 
\end{enumerate}
\end{proof}

\begin{lemma}\label{6asurj2}
There exists $\eta\in S$ which is coprime to $q_1\cdots q_k$ such that $\Phi[\eta^{-1}]$ is surjective.
\end{lemma} 
\begin{proof}
Take a dominant $\ll\in\Q$ such that $\a_i(w_0(\ll))$ is non-zero precisely when $1\leqslant i\leqslant k$. By Theorem \ref{3aPLS}, we have
\[ q^{-[w_0(\ll)]}\Phi_{PLS}([\agll]) = q^{-[w_0(\ll)]} \Phi_{PLS}(\xi_{t_{w_0(\ll)}}) = 1\]
and
\[ q^{-[w_0(\ll)]}\Phi_{PLS}\left(\bk{\ol{\BB\cdot t^{s_{\a_i}(w_0(\ll))}}}\right) = q^{-[w_0(\ll)]} \Phi_{PLS}(\xi_{s_{\a_i}t_{w_0(\ll)}})=\s_{s_{\a_i}}\]
for any $1\leqslant i\leqslant k$. Hence, by Corollary \ref{4bcor} and Corollary \ref{4bdivisor}, the image of the linear map
\[ \begin{array}{ccccc}
\FF&:&S(\ll)\otimes\R & \ra & \QHPi\\ [1em]
&&v\otimes 1 &\mapsto & q^{-[w_0(\ll)]}\Phi_{PLS}\circ\Phi_{YZ}(\tt(\JJ_{\ll},v,v_{\ll}^*))
\end{array} \]
contains $H_T^2(G/P)$. For any $v\in S(\ll)$, the regular function $s_{\ll,v}/s_{\ll,v_{\wp}}$ on $\Udd$ induces a regular function $f_v$ on $\YY$. We have $\vpii\circ\eti(f_v)=g_{P,\ll}^{-1}\cdot\tt(\JJ_{\ll},v,v_{\ll}^*)$. It follows that, by Lemma \ref{5bgp}, $\Phi(f_v)=\FF(v\otimes 1)$, and hence the image of $\Phi$ contains $H^2_T(G/P)$.

By \cite[Lemma 4.1.3]{Invent}, the $\fof\R$-algebra $H_T^{\bl}(G/P)\otimes_{\R}\fof\R$ is generated by $H_T^2(G/P)$. (We take a Pl\"ucker embedding of $G/P$ into $\PP(V)$ for some representation $V$ of $G$. The condition for this lemma is satisfied because different points of $(G/P)^T$ lie in different components of $\PP(V)^T$.) It follows that $(\Phi\otimes_S S/\langle q_1,\ldots,q_k\rangle)\otimes_{\R} \fof\R $ is surjective. By Nakayama's lemma, there exists $\eta\in S$ which is coprime to $q_1\cdots q_k$ such that $\Phi[\eta^{-1}]$ is surjective. 
\end{proof}

\begin{lemma}\label{6ainj}
$\Phi$ is injective.
\end{lemma}
\begin{proof}
Since $\OO(\YY)$ is a flat $S$-module by Lemma \ref{5aflat}, it suffices to show that $\Phi\otimes_S\fof S$ is injective. We have proved that it is surjective (Lemma \ref{6asurj1} or Lemma \ref{6asurj2}), and the dimension of its domain is equal to $|W/W_P|$ (Proposition \ref{5arank}) which is equal to the rank of the free $S$-module $\QHP$. The result follows.
\end{proof}

\vspace{.5cm}
\begin{myproof}{Proposition}{\ref{6athma}}
We have
\begin{enumerate}
\item $\OO(\YY)$ is a flat $S$-module (Lemma \ref{5aflat});

\item $\QHP$ is a flat $S$-module (obvious);

\item $\Phi[\ecc]$ is surjective (Lemma \ref{6asurj1});

\item $\Phi[\eta^{-1}]$ is surjective for some $\eta\in S$ which is coprime to $q_1\cdots q_k$ (Lemma \ref{6asurj2}); and

\item $\Phi$ is injective (Lemma \ref{6ainj}).
\end{enumerate}
By Lemma \ref{prooffact1}, (1) to (5) imply the bijectivity of $\Phi$.
\end{myproof}

\vspace{.5cm}
\begin{myproof}{Proposition}{\ref{6athmb}}
Recall $\Phi^*=\vpiii\circ\vpii$. Observe that $\OO(\YY^*)$ is the localization of $\OO(\YY)$ by $s_{\ll,v_e}/s_{\ll,v_{\wp}}$ where $\ll\in\Q$ is any regular dominant element. See Lemma \ref{2dlemma3}. It follows that the injectivity of $\Phi^*$ follows from the injectivity of $\Phi$ (Proposition \ref{6athma}). Moreover, we have 
\[ \im\Phi^*=\QHP[\Phi(s_{\ll,v_e}/s_{\ll,v_{\wp}})^{-1}] \]
by the surjectivity of $\Phi$ (Proposition \ref{6athma}). Notice that 
\[ \vpii\circ\eti(s_{\ll,v_e}/s_{\ll,v_{\wp}})^{-1}=g_{P,\ll}\cdot \tt(\JJ_{\ll},v_e,v_{\ll}^*)^{-1}. \]
Since $\Gd$ is of adjoint type, we have $\Bd_{e^T(0)}\subseteq\Ud$, and hence $\tt(\JJ_{\ll},v_e,v_{\ll}^*)|_{B^{\vee}_{e^T(0)}}=1$. On the other hand, we have $\vpiii(g_{P,\ll})=q^{[w_0(\ll)]}$ by Lemma \ref{5bgp}. Therefore, 
\[ \im\Phi^*|_{h=0} = QH^{\bl}(G/P)[q^{[w_0(\ll)]}]=QH^{\bl}(G/P)[\ecc].\]
\end{myproof}
\subsection{Proof of Theorem \ref{1athmb}} \label{6b}
Recall we have defined $\Phi_{loc}$ in Definition \ref{5bphi}, namely $\Phi_{loc}:=\vpiii$. By Lemma \ref{5bhomoloc}, it is a homomorphism of $S$-algebras. It remains to prove that it is bijective. Notice that $\Phi[\ecc]=\vpiii\circ\vpiv$ where $\vpiv$ is defined in the proof of Lemma \ref{6asurj1}. We proved there that $\vpiv$ is surjective. Since $\Phi[\ecc]$ is bijective by Proposition \ref{6athma}, it follows that $\vpiii=\Phi_{loc}$ and $\vpiv$ are bijective as well. Theorem \ref{1athmb} is proved.
\appendix  
\section{Some lemmas} \label{appA}
We prove a number of lemmas which are used in the preceding sections.
\begin{lemma} (Used in the proof of Corollary \ref{4bdivisor}) \label{4badded} 
Let $\ll\in\cop$ and $1\leqslant i\leqslant r$. Suppose $\a_i(w_0(\ll))\ne 0$. Then $\ol{\BB\cdot t^{s_{\a_i}(w_0(\ll))}}$ is the unique MV cycle of type $\ll$ and weight $w_0(\ll)+\ad_i$.
\end{lemma}
\begin{proof}
Put $\mu_0:=w_0(\ll)+\ad_i$. Let $Z$ be an irreducible component of $S^-_{\mu_0}\cap\agl$. Let $x\in Z$. Then there is unique $\mu\in W\cdot\ll$ such that $x\in\BB\cdot t^{\mu}$. Consider the loop rotation which is a $\gm$-action on $\ag$. Clearly, it commutes with the $T$-action and preserves $S^-_{\mu_0}\cap\agl$. Since $\BB\cdot t^{\mu}$ is the attractor with respect to a generic one-dimensional subtorus of $T\times\gm$, we have $t^{\mu}\in\ol{S^-_{\mu_0}\cap\agl}$, and hence $\mu\geqslant\mu_0$ by Lemma \ref{4bclosure}. This implies $\mu\ne w_0(\ll)$. It follows that $Z$ is contained in the union of $\ol{\BB\cdot t^{s_{\a_j}(w_0(\ll))}}$ where $1\leqslant j\leqslant r$ satisfies $\a_j(w_0(\ll))\ne 0$. Since $Z$ is irreducible and has codimension one in $\agll$ by Lemma \ref{4bdim}, it follows that $\ol{Z}$ is equal to $\ol{\BB\cdot t^{s_{\a_{j_0}}(w_0(\ll))}}$ for some $j_0$. To determine $j_0$, observe that $t^{s_{\a_{j_0}}(w_0(\ll))}\in\ol{S^-_{\mu_0}\cap\agl}$, and hence $s_{\a_{j_0}}(w_0(\ll))\geqslant\mu_0=w_0(\ll)+\ad_i$. But then we must have $j_0=i$. This establishes the uniqueness.

It remains to show $S^-_{\mu_0}\cap\agl\ne \emptyset$. Observe that $\BB\cdot t^{s_{\a_i}(w_0(\ll))}$ is covered by $S^-_{\mu}\cap\agl$ for finitely many $\mu$. It follows that it contains an open dense subset which is contained in $S^-_{\mu_1}\cap\agl$ for some $\mu_1$. Notice that $\mu_1\ne w_0(\ll)$. By a dimension argument, we see that $\mu_1=w_0(\ll)+\ad_{i_0}$ for some $1\leqslant i_0\leqslant r$ with $\a_{i_0}(w_0(\ll))\ne 0$, and hence by the discussion in the previous paragraph, we have $i_0=i$, or equivalently $\mu_0=\mu_1$. Therefore, $S^-_{\mu_0}\cap\agl=S^-_{\mu_1}\cap\agl\ne \emptyset$ as desired.
\end{proof}
\begin{lemma}\label{5aproperfact1} (Lemma \ref{5aproper}) Let $X$ be a scheme and $U$ an open subscheme which is the complement of the zero locus of a section of an invertible $\OO_X$-module. Let $Z$ be a closed subscheme of $U$. Suppose $Z$ is a closed subset of $X$. Then $Z$ is a closed subscheme of $X$.
\end{lemma}
\begin{proof}
Clearly it suffices to deal with the affine case so we assume $X=\spec A$, $U=\spec A[f^{-1}]$ for some $f\in A$ and $Z=\spec A[f^{-1}]/I$ for some ideal $I$ of $A[f^{-1}]$. We have to show that the composition
\[ A \ra A[f^{-1}]\ra A[f^{-1}]/I \]
of the canonical ring homomorphisms is surjective. Denote by $I'$ the preimage of $I$ under the first map. It suffices to show that $f$ becomes a unit in $A/I'$. Suppose to the contrary. Then there is a prime ideal $\mathfrak{p}$ of $A$ such that $I'\subseteq \mathfrak{p}$ and $f\in \mathfrak{p}$. The second condition implies that $\mathfrak{p}\not\in Z$. Since $Z$ is closed in $X$, there is $g\in A$ such that $Z$ is contained in the closed subset defined by $g$ and $\mathfrak{p}$ lies in the complement of this subset. It follows that $g\in I'$ and $g\not\in \mathfrak{p}$, in contradiction to $I'\subseteq \mathfrak{p}$.
\end{proof}
\begin{lemma}\label{5aproperfact2} (Lemma \ref{5aproper}) Let $\{\od_1,\ldots,\od_r\}$ be the dual basis of $\{\a_1,\ldots,\a_r\}$. Let $0\leqslant \ell < r$. There exist non-negative rational numbers $b_1,\ldots,b_{\ell}$ and $c_{\ell+1},\ldots,c_r$ such that 
\begin{equation}\label{5aproperfact2eq1}
\od_{\ell+1} = \sum_{i=1}^{\ell} b_i\od_i +\sum_{j=\ell+1}^rc_j\ad_j.
\end{equation}
\end{lemma}
\begin{proof}
Put $\mathbf{b}:=(b_1,\ldots,b_{\ell})^T$ and $\mathbf{c}:=(c_{\ell+1},\ldots,c_r)^T$. By pairing \eqref{5aproperfact2eq1} with the simple roots, we obtain the following linear system
\[
\begin{pmatrix}
\id & C_1\\ 0& C_2
\end{pmatrix}
\begin{pmatrix}
\mathbf{b}\\ \mathbf{c}
\end{pmatrix}
= \begin{pmatrix}
0\\ e_1
\end{pmatrix}
\]
for some matrices $C_1$ and $C_2$, where $e_1:=(1,0,\ldots,0)^T$. It follows that 
\[
\begin{pmatrix}
\mathbf{b}\\ \mathbf{c}
\end{pmatrix}
=\begin{pmatrix}
\id & -C_1C_2^{-1}\\ 0& C_2^{-1}
\end{pmatrix}
\begin{pmatrix}
0\\ e_1
\end{pmatrix}
= \begin{pmatrix}
-C_1C_2^{-1}e_1\\ C_2^{-1}e_1
\end{pmatrix}.
\]
Observe that $C_1$ is a submatrix of the Cartan matrix for the root system spanned by $\{\a_1,\ldots,\a_r\}$ and $C_2$ is the Cartan matrix for the root system spanned by $\{\a_{\ell+1},\ldots,\a_r\}$. The result follows from 
\begin{enumerate}
\item all non-diagonal entries of a Cartan matrix are $\leqslant 0$; and 

\item all entries of the inverse of a Cartan matrix are $\geqslant 0$.
\end{enumerate}
(1) is well-known. (2) is equivalent to 
\begin{enumerate}
\item[(3)] the dominant chamber is contained in the closed cone spanned by the simple roots.
\end{enumerate}
See \cite[\S 3 no.5 Lemme 6]{Bourbaki} for a proof of (3).
\end{proof}
\begin{lemma}\label{5bexist1fact1} (Lemma \ref{5bexist1}) Let $\ll\in\Q$ be dominant and $\mu_1,\ldots,\mu_N\in\Q\cap \conv(W\cdot\ll)$. We have
\[\sum_{i=1}^N(\mu_i-\wp(\ll))\in \left(\sum_{\a\in R^+_P}\ZZ_{\leqslant 0}\cdot\ad\right)\sm\{0\}\Longleftrightarrow \mu_i< \wp(\ll)~\text{ for all }i. \]
\end{lemma}
\begin{proof}
The condition $\mu_i\in\Q\cap \conv(W\cdot\ll)$ implies 
\[ \mu_i-\wp(\ll)\in \left(\sum_{j=1}^k \ZZ_{\geqslant 0}\cdot w_P\ad_j\right) +\left(\sum_{j=k+1}^r \ZZ_{\leqslant 0}\cdot \ad_j\right)  .  \]
(Recall $R_P\cap \{\a_1,\ldots,\a_r\}=\{\a_{k+1},\ldots,\a_r\}$.) The rest is clear.
\end{proof}
\begin{lemma}\label{prooffact1} (Proposition \ref{6athma}) Let $A$ be a unique factorization domain and $a_1,a_2\in A$ two non-zero coprime elements. Let $M$ and $N$ be two flat $A$-modules and $\Phi:M\ra N$ an $A$-linear map. Suppose $\Phi$ is injective and $\Phi[a_i^{-1}]$ is surjective for $i=1,2$. Then $\Phi$ is bijective.
\end{lemma}
\begin{proof}
Since $a_1$ is coprime to $a_2$, we have an exact sequence
\[ 0 \ra A \ra A[a_1^{-1}]\oplus A[a_2^{-1}]\ra A[a_1^{-1},a_2^{-1}]\]
of $A$-modules. Apply the natural transformation $P\mapsto \Phi\otimes_A\id_P$ and five lemma.
\end{proof}

\section{Signs} \label{app}
In this appendix, we use the generators $e_{\ad_i}\in\dd{\gg}_{\ad_i}$ from Section \ref{4a} to construct explicit representatives $\dot{w}_0$ and $\dot{w}_P$, and settle the sign issues occurring in the proofs of Lemma \ref{5aq=1} and Lemma \ref{5bhomo}. Unless otherwise specified, the coefficient ring is $\ZZ$.

For any $\ll\in\cop$, $\mu\in\co$ and MV cycle $Z$ of type $\ll$ and weight $\mu$, denote by $v_Z\in S(\ll)_{\mu}$ the generator corresponding to $Z$ via the isomorphism in Lemma \ref{2clemma2}.
\begin{lemma}\label{applemma1}
Let $1\leqslant i\leqslant r$ and $Z$ be an MV cycle of type $\ll$ and weight $\mu$. Consider the $\dd{\gg}$-action on $S(\ll)$. Write
\[ e_{\ad_i}\cdot v_Z = \sum_{Z'} c_{Z'} v_{Z'}\]
where each $c_{Z'}\in\ZZ$ and the sum is taken over all MV cycles $Z'$ of type $\ll$. Then $c_{Z'}\geqslant 0$ for all $Z'$. 
\end{lemma}
\begin{proof}
We may assume the coefficient ring is $\CC$. The result follows immediately from \cite[Theorem 5.4]{Acta} which says that each $Z'$ for which $c_{Z'}\ne 0$ must be an irreducible component of the support of a Cartier divisor $D$ on $Z$, and in this case $c_{Z'}$ is equal to $|\ad_i|^{-2}$ times the intersection multiplicity of $Z'$ along $D$. 
\end{proof}

We now construct specific $\dot{w}_0$ and $\dot{w}_P$. The following approach is standard. Let $\a\in R$. By \cite[Theorem 4.1.4]{BC}, there is a unique homomorphism
\[ \exp_{\ad}:\spec\sym^{\bl}(\dd{\gg}_{\ad})^*\ra \Gd\]
of group schemes such that $\lie(\exp_{\ad})$ is the inclusion $\gg^{\vee}_{\ad}\hookrightarrow\gg^{\vee}$. By \cite[Theorem 4.2.6 \& Corollary 5.1.12]{BC}, there is a perfect pairing 
\[ \b_{\ad}:\dd{\gg}_{\ad}\otimes \dd{\gg}_{-\ad} \ra \ZZ\]
satisfying
\[ [X,Y] = \b_{\ad}(X,Y) H_{\ad}\]
for any $X\in\dd{\gg}_{\ad}$ and $Y\in\dd{\gg}_{-\ad}$, where $H_{\ad}:=\lie(\a)(1)\in\lie(\Td)$. For each $1\leqslant i\leqslant r$, denote by $e^{-1}_{\ad_i}\in \dd{\gg}_{-\ad_i}$ the unique element such that $\b_{\ad_i}(e_{\ad_i},e^{-1}_{\ad_i})=1$.

\begin{definition}\label{appdef}$~$
\begin{enumerate}
\item Let $1\leqslant i\leqslant r$. Define
\[ \dot{s}_{\ad_i} := \exp_{-\ad_i}(-e^{-1}_{\ad_i})\exp_{\ad_i}(e_{\ad_i})\exp_{-\ad_i}(-e^{-1}_{\ad_i}).\]

\item Let $w\in W$. Define
\[ \dot{w}:=\dot{s}_{\ad_{i_1}} \cdots \dot{s}_{\ad_{i_{\ell(w)}}} \]
where $w=s_{\a_{i_1}}\cdots s_{\a_{i_{\ell(w)}}}$ is a reduced word decomposition of $w$. 
\end{enumerate}
\end{definition}

\begin{lemma}\label{applemma2}
Each $\dot{w}$ is a $\ZZ$-point of the normalizer $N_{\Gd}(\Td)$ and represents $w$ in the quotient $N_{\Gd}(\Td)/\Td\simeq W$. It is independent of the choice of the reduced word decomposition of $w$. In particular, if $w_1,w_2\in W$ satisfy $\ell(w_1w_2)=\ell(w_1)+\ell(w_2)$, then $\dot{(w_1w_2)}=\dot{w}_1\dot{w}_2$. 
\end{lemma}
\begin{proof}
The first assertion follows from \cite[Corollary 5.1.9]{BC}. For any two reduced word decompositions of $w$, the corresponding two $\dot{w}$ are related by a $\ZZ$-point of $\Td$ which must be the unit by Lemma \ref{applemma4} below. This proves the second assertion. The last assertion follows from the second assertion. 
\end{proof}

\begin{lemma}\label{applemma3}
We have
\[ \dot{w}_0^2 = \left(\sum_{\a\in R^+}\a\right)(-1)\quad \text{ and }\quad\dot{w}_P^2 = \left(\sum_{\a\in R^+_P}\a\right)(-1).\]
In particular, $\dot{w}^2_0=e$ if $\Gd$ is of adjoint type.
\end{lemma}
\begin{proof}
Let $w_0=s_{\a_{i_1}}\cdots s_{\a_{i_{\ell(w_0)}}}$ be a reduced word decomposition of $w_0$. Then $w_0=w_0^{-1}=s_{\a_{i_{\ell(w_0)}}}\cdots s_{\a_{i_1}}$ is also a reduced word decomposition of the same element. By Lemma \ref{applemma2},
\[\dot{w}_0^2 =\dot{s}_{\ad_{i_1}} \cdots \dot{s}_{\ad_{i_{\ell(w_0)}}} \dot{s}_{\ad_{i_{\ell(w_0)}}}\cdots \dot{s}_{\ad_{i_1}} .\]
The first equality then follows from the identity $\dot{s}_{\ad_{i}}^2=\a_i(-1)$ (see \cite[Corollary 5.1.11]{BC}) and a straightforward computation. The second equality is proved similarly.
\end{proof}

Let $\ll\in\cop$ and $w\in W$. Denote by $v_w\in S(\ll)_{w(\ll)}\simeq \ZZ$ the generator corresponding to the unique MV cycle of type $\ll$ and weight $w(\ll)$ via the isomorphism in Lemma \ref{2clemma2}.
\begin{lemma}\label{applemma4}
We have $\dot{w}\cdot v_{w_0} = v_{ww_0}\in S(\ll)_{ww_0(\ll)}$.
\end{lemma}
\begin{proof}
It suffices to show that $\dot{s}_{\ad_i}\cdot v_u=v_{s_{\a_i}u}$ for any $1\leqslant i\leqslant r$ and $u\in W$ such that $\ell(s_{\a_i}u)=\ell(u)-1$. Clearly, we have $\dot{s}_{\ad_i}\cdot v_u=\pm v_{s_{\a_i}u}$. To determine the sign, we may work over $\CC$ so that we can expand the exponential maps involved in $\dot{s}_{\ad_i}$ (see Definition \ref{appdef}). But by the assumption $\ell(s_{\a_i}u)=\ell(u)-1$, only the middle factor $\exp_{\ad_i}(e_{\ad_i})$ contributes to the sign, and hence the sign is positive by Lemma \ref{applemma1}.
\end{proof}

Finally, the following propositions will settle the sign issues occurring in the proofs of Lemma \ref{5aq=1} and Lemma \ref{5bhomo}.

\begin{proposition}\label{appprop1}
For any $k +1\leqslant i\leqslant r$, we have
\begin{equation}\label{appeq}
\dot{w}_P\dot{w}_0\cdot e_{-w_0\ad_i}=e_{-w_P\ad_i}.
\end{equation}
\end{proposition}
\begin{proof}
Clearly, both sides of \eqref{appeq} are equal up to sign so it suffices to show that the sign is $+1$. We may assume $\Gd$ is simply-connected so that we can consider the representation $S(\od_i)$ where $\od_i$ is the fundamental coweight corresponding to the simple root $\a_i$. Put $\a_j:=-w_0\a_i$. We have
\[ e_{\ad_j}\cdot v_{w_0} = \dot{s}_{\ad_j}\cdot v_{w_0}\in S(\w_i^{\vee})_{w_0(\w_i^{\vee})+\ad_j},\]
and hence 
\[ (\dot{w}_P\dot{w}_0\cdot e_{\ad_j})\cdot (\dot{w}_P\dot{w}_0\cdot v_{w_0}) = \dot{w}_P\dot{w}_0\dot{s}_{\ad_j}\cdot v_{w_0}.\]

Write 
\[\dot{w}_P\dot{w}_0\cdot v_{w_0} = m_1 v_{w_P}\quad\text{ and }\quad \dot{w}_P\dot{w}_0\dot{s}_{\ad_j}\cdot v_{w_0} =m_2 v_{w_Pw_0s_{\a_j}w_0}\]
where $m_1,m_2\in\{\pm 1\}$. By Lemma \ref{applemma1}, the result will follow if we can show $m_1=m_2$. Let $u$ denote $e$ or $s_{\a_j}$ and put $u':=w_Pw_0u$. By Lemma \ref{applemma2}, we have
\[ \dot{(u')}=\dot{(w_Pw_0)}\dot{u} = \dot{w}_P^{-1}\dot{w}_0\dot{u}=\dot{w}_P^{-2}(\dot{w}_P\dot{w}_0\dot{u}).\]
By Lemma \ref{applemma4}, we have $v_{u'w_0}=\dot{(u')}\cdot v_{w_0}$. It follows that
\[ v_{w_P} =\dot{w}_P^{-2}(\dot{w}_P\dot{w}_0)\cdot v_{w_0}\quad\text{ and }\quad v_{w_Pw_0s_{\a_j}w_0}= \dot{w}_P^{-2}(\dot{w}_P\dot{w}_0\dot{s}_{\ad_j})\cdot v_{w_0},  \]
and hence, by Lemma \ref{applemma3},
\[ m_1/m_2 = (-1)^{\left(\sum_{\a\in R_P^+}\a\right)[ w_P(\od_i)-w_Ps_{\a_i}(\od_i)]} = (-1)^{\left(\sum_{\a\in R_P^+}\a\right)(w_P(\ad_i))}=(-1)^{-2}=1.\]
\end{proof}

\begin{proposition}\label{appprop2}
Suppose $\Gd$ is of adjoint type. For any $\ll\in\cop$, we have
\[ v_{w_Pw_0} = \dot{w}_P\dot{w}_0\cdot v_e\in S(\ll)_{\wp(\ll)}.\]
\end{proposition}
\begin{proof}
By Lemma \ref{applemma4},
\[ v_{w_Pw_0} = \dot{w}_P\cdot v_{w_0} = \dot{w}_P\dot{w}_0^{-1}\dot{w}_0\cdot v_{w_0}=\dot{w}_P\dot{w}_0^{-1}\cdot v_e. \]
Since $\Gd$ is of adjoint type, we have, by Lemma \ref{applemma3}, $\dot{w}_0^{-1}=\dot{w}_0$. The result follows.
\end{proof}

\end{document}